\newif\ifnotarxiv
\notarxivfalse
\documentclass{amsart}
\newtheorem{theorem}{Theorem}
\newtheorem{lemma}{Lemma}
\author{Kristian Debrabant and Anne Kv{\ae}rn{\o} and Nicky Cordua Mattsson
}
\newcommand{\subclass}[1]{\subjclass{#1}}
\address[Kristian Debrabant]{Department of Mathematics and Computer Science, University of Southern Denmark, 5230 Odense M, Denmark}
\email{debrabant@imada.sdu.dk}
\address[Anne Kv{\ae}rn{\o}]{Department of Mathematical Sciences, Norwegian University of Science and Technology, 7491 Trondheim, Norway}
\email{anne.kvarno@math.ntnu.no}
\address[Nicky Cordua Mattsson]{Department of Mathematics and Computer Science, University of Southern Denmark, 5230 Odense M, Denmark}
\email{mattsson@imada.sdu.dk}
\newif\ifUpdateFigures\UpdateFiguresfalse
\usepackage{mathptmx}       % selects Times Roman as basic font
\usepackage{helvet}         % selects Helvetica as sans-serif font
\usepackage{courier}        % selects Courier as typewriter font
\usepackage{type1cm}        % activate if the above 3 fonts are
\usepackage{graphicx}        % standard LaTeX graphics tool
\usepackage{float}
\usepackage{hyperref}
\usepackage[all]{hypcap}
\usepackage{microtype}
\usepackage{newtxtext,newtxmath}

\usepackage{amsmath}            %MATH!
\usepackage{mathtools}
\usepackage{mathptmx}
\usepackage{amsfonts}        %More MATH
\usepackage{nicefrac}
\newtheorem{assumption}{Assumption}

\usepackage{enumerate}
\usepackage{algorithm}
\usepackage{algorithmic} %For algorithms
\usepackage{pgfplots}
\usepackage{subcaption}
\usepgfplotslibrary{fillbetween}
\pgfplotsset{compat=newest,every axis/.append style={legend style={draw=none,fill=none,font=\tiny}}}
\allowdisplaybreaks
\usepackage{pgfplotstable}
\ifUpdateFigures\else\renewcommand{\pgfplotstableread}[3][]{}\fi
\usepgfplotslibrary{external}
\ifUpdateFigures\else\fi
\tikzsetfigurename{ExponentialSDE-figure}
\usepackage[capitalize]{cleveref}
\Crefname{assumption}{Assumption}{Assumptions}
\crefname{inequality}{inequality}{inequalities}
\creflabelformat{inequality}{(#2#1#3)}
\crefname{item}{}{}
\creflabelformat{item}{(#2#1#3)}
\crefname{equation}{}{}%{eq.}{eqs.}

\usepackage{csquotes}

\usepackage{scrtime}

\newcommand{\Ih}{I^h}
\newcommand{\gh}{\hat{g}}
\newcommand{\gb}{\bar{g}}
\newcommand{\gt}{\tilde{g}}
\newcommand{\R}{\mathbb{R}}
\newcommand{\N}{\mathbb{N}}
\newcommand{\C}{\mathbb{C}}
\newcommand\gO{\mathcal{O}}
\newcommand{\dmath}{\mathrm{d}}
\newcommand{\dt}{\dmath t}
\newcommand{\ds}{\dmath s}
\newcommand{\dW}{\dmath W}
\newcommand{\dX}{\dmath X}
\newcommand{\dV}{\dmath V}
\newcommand{\Id}{\mathbb{I}_d}
\newcommand{\Ad}{\mathfrak{A}}
\newcommand{\Bd}{\mathfrak{B}}
\newcommand{\Smet}{\mathfrak{S}}
\newcommand{\gs}{\gamma^\star}
\newcommand{\ito}{It\^{o}}

\DeclareMathOperator{\E}{E}

\newcommand{\hs}{\varphi}
\usepackage{todonotes,soul,csquotes}

\makeatletter
\renewcommand{\todo}[2][]{\tikzexternaldisable \@todo[#1]{#2}\tikzexternalenable}
\makeatother

\makeatletter
\if@todonotes@disabled

\else

\fi
\makeatother

\begin{document}

\title{Runge--Kutta Lawson schemes for stochastic differential equations}

\ifnotarxiv\maketitle\fi

\begin{abstract}
In this paper, we present a framework to construct general stochastic Runge--Kutta Lawson schemes.
We prove that the schemes inherit the consistency and convergence properties of the underlying
Runge--Kutta scheme, and confirm this in some numerical experiments. We {also investigate} the stability properties of the methods and show for some examples, that the new schemes have improved stability properties compared to the underlying schemes.
\subclass{ 60H35 \and 60H10 \and 65L20 \and 93E15}
\end{abstract}

\maketitle

\section{Introduction}
Stochastic differential equations {(SDEs)} are an essential tool in order to model and understand real-life systems under the influence of noise, see, e.g., \cite[Section 7.1-7.10]{kloeden99nso} for examples. We do, however, only know the exact solution to very few of these equations, thus will usually have to integrate the differential equations numerically. To efficiently do this, it is of interest to construct numerical schemes that can recreate essential dynamics of the exact solution.

Linear terms in the drift and the diffusion can often represent these essential dynamics. See e.g.\ recent works on linear stability analysis \cite{buckwar12nds,buckwar10tas,buckwar11acl} and highly oscillatory differential equations \cite{cohen12otn,debrabantXXlsf}. For the same reason, much work has also gone into treating these parts explicitly, see e.g.\  the work on local linearization techniques by \cite{biscay96llm,jimenez99sos,carbonell09wll}. In this paper we assume that the relevant linear terms have been made explicitly available, and thus consider SDEs of the form
\begin{equation}\label{equ:sde}
\dX(t)= \sum_{m=0}^{M} (A_m X(t) + g_m(t,X(t)){)}\star\dW_m(t),\quad X(t_0)=X_{0},
\end{equation}
where $W_m$ for $m=1,\ldots,M$ denote independent scalar Wiener processes, $W_0{(t)} = t$ denotes the time and the SDE is solved on the interval ${I} = [t_0,T]$. {Here, the stochastic integral can be interpreted as Itô integral with $\star\dW_m=\dW_m$, or as Stratonovich integral with $\star\dW_m=\circ\dW_m$.} We assume that SDE \labelcref{equ:sde} has a unique solution for $X_0\in \mathbb{R}^d$ and that all {$g_m$} have the appropriate regularity for this (depends on the interpretation of the integral). We also assume that the matrices $A_m \in \mathbb{R}^{d\times d}$, $m=0,\ldots,M$, are constant and {are chosen in connection with $g_m$} such that the following assumption {holds}:
\begin{assumption}[Commutativity]\label{ass:commute}
  \[ [A_l,A_k] =A_lA_k - A_kA_l = 0  \qquad \text{for all} \qquad l,k = 0,1,\dots,M. \]
\end{assumption}%

Exponential integrators have, especially in the deterministic case, been very efficient at solving some types of differential equations. In the more recent years, much work has gone into extending these results and schemes to {SDEs}. In particular, Erdogan and Lord \cite{erdogan19anc} construct an exponential Euler and an exponential Milstein scheme, including both the linear drift and diffusion, and numerically demonstrate that these schemes are more efficient on specific problems than their underlying schemes. They also show that for linear diffusion, the strong order of convergence of the exponential Euler scheme is $p=1$. We also mention the work by \cite{komori14ase,komori17wso}, who apply specific exponential schemes to a stiff system and construct an explicit weak second-order exponential scheme that proves to be A-stable for the linear test-equation suggested by \cite{higham00msa}.

In this paper, we construct exponential integrators including both the linear drift and diffusion using the entire class of stochastic Runge-Kutta (SRK) schemes and a stochastic extension of Lawson type schemes (also known as integrating factor methods) for both Stratonovich and \ito\ integrals. We prove that these stochastic Lawson (SL) schemes, under some conditions, inherit both the strong and weak order of convergence of the underlying scheme, and provide a general framework to construct higher-order exponential schemes.

In \cref{sec:Lawsons} we present a stochastic extension of the deterministic Lawson transformation \cite{lawson67grk} and derive then the general class of SRK Lawson schemes, providing also several examples. In \cref{sec:SRK}, we prove that these methods, under some conditions, converge both strongly and weakly of the same order as the underlying SRK schemes. These results are accompanied by numerical examples. In \cref{sec:MSstab} we provide some linear stability analyses for a selection of these schemes, showing that exponential SRK schemes may have improved stability properties. We also compare these methods to the drift implicit Platen scheme, where we show that the Platen Lawson scheme better catches the behaviour of the {reference} solution {for the considered examples}. These results are also verified by numerical simulations.

\section{Construction of {SRK} Lawson schemes} \label{sec:Lawsons}
In this section, we present the overall idea of Lawson type schemes; we then apply the idea to the class of {SRK} schemes and provide several examples.
\subsection{General construction}
This section aims at constructing a numerical scheme which solves linear SDEs of the form
\begin{equation}\label{equ:linSDE}
    \dX(t)=\sum_{m=0}^M A_m X(t)\star \dW_m(t), \quad X(t_0)=X_0
\end{equation}
exactly. Under \cref{ass:commute}, the exact solution of \eqref{equ:linSDE} can be written as \cite{arnold74sde,erdogan19anc}
\begin{equation}\label{equ:linSDEcommsol}
    X(t)=\exp\left[\left(A_0 - \gs \sum_{m=1}^{M} A^2_m \right)(t-t_0) + \sum_{m=1}^{M} A_m (W_m(t)-W_m(t_0))\right] X_0,
\end{equation}
with $\gs=\frac{1}{2}$ in the \ito\ case, $\gs=0$ in the Stratonovich case (see also \cref{lem:transformation}). This exact solution will be used to construct the exponential integrators, in accordance to the approach to construct deterministic Lawson schemes \cite{lawson67grk}.

Before we do so, we want to emphasize that \cref{ass:commute} is not a restriction on the SDE to be considered, but rather a restriction on how to pick the matrices $A_m$. For demonstration consider the SDE
\begin{equation}\label{eq:GenSDE}
\dX(t)=\left(\tilde{A}_0 X(t) + f_0(t,X(t))\right)\dt+ \left(\tilde{A}_1 X(t) + f_1(t,X(t))\right)\star\dW(t),
\end{equation}
where $\tilde{A}_0$ does not commute with $\tilde{A}_1$. It is always possible to find a splitting of $\tilde{A}_{{0}}$,
\begin{equation*}
    \tilde{A}_0 = A_0 + A_0^r,
\end{equation*}
such that $A_0$ commutes with $\tilde{A}_1$, and we can thus rewrite \cref{eq:GenSDE} in the form of \cref{equ:sde} by choosing
\begin{equation*}
    g_0(t,X(t)) = f_0(t,X(t)) + A_0^r X(t),\qquad g_1(t,X(t))=f_1(t,X(t)),\qquad A_1=\tilde{A}_1.
\end{equation*}
An obvious choice is $A_0=c\tilde{A}_1$ for some {scalar} $c$. Optimally $A_0$ should represent the properties of the SDE that we are interested in simulating exactly. Conversely, we can also put $A_1=c_1\tilde{A}_0$, which ensures that $A_1$ commutes with $\tilde{A}_0$.

Let a discretization ${\Ih} = \{t_0, t_1, \ldots, t_N\}$ with $t_0 < t_1 < \ldots < t_N =T$ of the time interval $I$ be given with $h_n = t_{n+1}-t_n$ for $n=0,1, \ldots, N-1$ denoting the step size. To construct numerical schemes that solve \cref{equ:linSDE} exactly, we will use the following \lcnamecref{lem:transformation}.
 \begin{lemma}\label{lem:transformation}
     Let $X$ be the solution of SDE \cref{equ:sde} and let \cref{ass:commute} be true. Then the locally transformed variable $V^n$ defined by
     \begin{equation}\label{equ:TransformedVar}
         V^n(t) = e^{-L^{n}(t)} X(t)
     \end{equation}
     with
     \begin{equation}\label{eq:Lndef}
         L^n(t) = \left(A_0-\gs\sum_{m=1}^M A_m^2\right)(t-t_n) + \sum_{m=1}^M A_m(W_m(t) - W_m(t_n))
     \end{equation}
     satisfies the SDE
     \begin{equation} \label{equ:TransformedSDE}
  \dV^n(t) =\sum_{m=0}^M \gh^n_m(W(t),V^n(t))\star\dW_m(t), \quad V^n(t_n)=X(t_n),
     \end{equation}
     where $W(t)=(W_m(t))_{m=0}^M$ and
     \begin{align*}
         \gh^n_m(W(t),x) :=& e^{-L^n(t)} \gt_m(t,e^{L^n(t)}x)
     \end{align*}
     with
     \begin{align*}
         \gt_m(t,x):= \begin{cases}
             g_0(t,x){-}2\gs \sum_{m=1}^M A_mg_m(t,x), & m=0, \\
            g_m(t,x), & m>0.
         \end{cases}
     \end{align*}
 \end{lemma}
 \begin{proof}
{Using \cref{ass:commute} when applying \ito's lemma  to the transformed variable $V^n$ defined in \cref{equ:TransformedVar} yields}{
\begin{multline*}
\dV^n(t) =-\left(A_0-\gs\sum_{m=1}^M A_m^2\right)e^{-L^{n}(t)} X(t)\dt-\sum_{m=1}^M A_me^{-L^{n}(t)} X(t)\dW_m\\
+\gs\sum_{m=1}^M A_m^2e^{-L^{n}(t)} X(t)\dt+e^{-L^{n}(t)}\dX(t)\\+2\gs\sum_{m=1}^M{\Big{(}}-A_me^{-L^{n}(t)}{\Big{)}}{\Big(}A_m X(t) + g_m(t,X(t)){\Big)}\dt.
\end{multline*}
Inserting \cref{equ:sde} we obtain (again using \cref{ass:commute})
\begin{align*}
\dV^n(t) &={-}2\gs e^{-L^{n}(t)}\sum_{m=1}^M A_mg_m(t,X(t))\dt+e^{-L^{n}(t)}\sum_{m=0}^{M} g_m(t,X(t)) \star\dW_m(t)
\end{align*}
which together with \cref{equ:TransformedVar} yields the assertion.}
 \end{proof}
{Note that by \cref{lem:transformation}, every solution of \cref{equ:sde} induces a solution of \cref{equ:TransformedSDE}. Analogously one can also show that every solution $V^n$ of \cref{equ:TransformedSDE} induces a solution $X(t)=e^{L^n(t)}V^n(t)$ of \cref{equ:sde}. As \cref{equ:sde} is assumed to have a unique solution, this holds thus also for  \cref{equ:TransformedSDE}.}

Note {further} that for linear SDEs \eqref{equ:linSDE}, the right hand side of \eqref{equ:TransformedSDE} will vanish. Denoting by $Y_n$ the discrete-time approximation to $X(t_n)$, we now define the family of one-step SL schemes by \cref{alg:LocalLawson}.

\begin{algorithm}[ht]
\caption{Stochastic Lawson scheme}\label{alg:LocalLawson}
\begin{enumerate}[(a)]
  \item Start with $Y_0=x_0$.
  \item For $n=0,\dots,N-1$
  \begin{enumerate}[(i)]
    \item \label{algparta}Apply one step of a one-step method to get an approximation $V^n_{n+1}$ of the exact solution $V^n(t_{n+1})$ of
\begin{equation}\label{equ:TransformedSDELocal}
\dV^n(t) = \sum_{m=0}^M \gh^n_m(W(t),V^n(t))\star\dW_m(t), \quad V^n(t_n)=Y_n
\end{equation}
    at $t_{n+1}$.
    \item Define
\begin{equation*}
    Y_{n+1}=e^{\Delta L^n}V^n_{n+1}
\end{equation*}
where $\Delta L^n{=\left(A_0-\gs\sum_{m=1}^M A_m^2\right)(t_{n+1}-t_n) + \sum_{m=1}^M A_m\Delta W_m^n\in\R^{d\times d}}$ {with
$\Delta W_m^n$ being the approximation of $W_m(t_{n+1}) - W_m(t_n)$ used by the one-step method applied in step (i).}
  \end{enumerate}
\end{enumerate}
\end{algorithm}
We note that if $L^n(t)$ is a stochastic process or if we use variable step-sizes, $e^{L^n(t)}$ has
to be recalculated at every step, which might be expensive for large dimensions $d$. One way to
avoid doing this is to pick $A_m = 0$ for $m>0$ and to use constant step sizes, $h_n=h$; then one
only has to calculate the exponentials once. We denote such schemes as drift SL (DSL) schemes. In
contrast, a full stochastic Lawson (FSL) scheme is a scheme with at least one nonzero linear
diffusion term present in the operator ${L^n(t)}$. The underlying numerical scheme is restored by
setting all the linear parts to 0. This is typically the case if the linear part is a result of a
spatial discretization of a diffusion term.
{Compared to the original schemes the computational overhead of the DSL schemes is negligible, as
  the matrix exponential is deterministic and is calculated only once. For the FSL schemes, this is not the
  case. Thus the applicability of the FSL methods depends on how efficient the matrix exponentials can
be calculated, weighted with the advantages of improved performance. See \cite{cohen12otn,debrabantXXlsf} for some successful examples.}

We now apply this general idea to the class of {SRK} schemes.

\subsection{{SRK} Lawson schemes}
To apply an SRK method to \cref{equ:TransformedSDE} we first have to transform \cref{equ:TransformedSDE} to an autonomous system, i.e.\ add $M+1$ equations in order to obtain an autonomous SDE in terms of $\bar{V}^n(t) = (W(t)^\top,{V^n}(t)^\top)^{\top}$,
\begin{equation}\label{equ:autonomousSDE}
    \dmath \bar{V}^n(t) = \sum_{m=0}^M \gb^n_m(\bar{V}^n(t)) \star \dW_m(t), \quad \gb^n_m = (\delta_{0,m},\ldots,\delta_{M,m},{\gh^{n\top}_m})^{\top}, \mbox{ } \delta_{i,j} = \begin{cases}
        1, & i=j, \\
        0, & i\neq j.
    \end{cases}
\end{equation}
Using the same notation as in \cite{burrage00oco}, an $s$-stage SRK method applied to \cref{equ:autonomousSDE} is given by
\begin{equation} \label{equ:SRK}
    \begin{aligned}
        \bar{H}_i &= \bar{V}^n_n + \sum_{j=1}^{s} \sum_{m=0}^M Z_{ij}^{m,n} \gb^n_m(\bar{H}_j), \\
        \bar{V}^n_{n+1} &= \bar{V}^n_n + \sum_{i=1}^{s} \sum_{m=0}^M z_i^{m,n} \gb^n_m(\bar{H}_i)
    \end{aligned}
\end{equation}
with suitable random variables $Z_{ij}^{m{,n}}$ and $z_i^{m{,n}}$.
Letting
\[c^{n,i}_m = \sum_{j=1}^s Z_{ij}^{m,n},\qquad c^n_m = \sum_{i=1}^s z_i^{m,n}\]
it follows for the approximations $V^n_{n+1}$ to $V^n(t_{n+1})$ and $W^n$ to $W(t_n)$ and the corresponding stage values that
\begin{equation}\label{eq:sRKVorig}
\begin{aligned}
    H_i &= V^n_n + \sum_{j=1}^s \sum_{m=0}^M Z_{ij}^{m,n} \gh^n_m(W^n+c^{n,j}, H_j), \\
    W_{m}^{n+1}&=W_{m}^n+c^n_m,\qquad {m=0,\dots,M,} \\
    V^n_{n+1} &= V^n_n + \sum_{i=1}^s \sum_{m=0}^M z_i^{m,n} \gh^n_m(W^n+c^{n,i}, H_i),
\end{aligned}
\end{equation}
where $W^n+ c^{n,i} = (W^n_m+c^{n,i}_m)_{m=0}^M$. Defining the discrete updates $\Delta W^n_m=W_{m}^{n+1}-W_{m}^n=c^n_m$,
\begin{equation*}
    \Delta L^n_i = \left(A_0 - \gs \sum_{m=1}^M A_m^{{2}}\right)c^{n,i}_{{0}} + \sum_{m=1}^M A_m c^{n,i}_m,
\end{equation*}
and
\begin{equation}\label{eq:DeltaLndef}
     \Delta L^n = \left(A_0 - \gs \sum_{m=1}^M A_m^{{2}}\right)\Delta W^n_0 + \sum_{m=1}^M A_m \Delta W^n_m
\end{equation}
and using the particular form of $\hat{g}$ and transforming back{,} we get the family of {SRK} Lawson schemes
\begin{equation}\label{equ:SRKL}
\begin{aligned}
    H_i &= Y_n + \sum_{j=1}^s \sum_{m=0}^M Z_{ij}^{m,n} e^{-\Delta L_j^{n}}\gt_m(t_n+c_0^{n,j},e^{\Delta L_j^n} H_j), \\
    V^n_{n+1} &= Y_n + \sum_{i=1}^s \sum_{m=0}^M z_i^{m,n} e^{-\Delta L_i^n}\gt_m(t_n+c^{n,i}_0,e^{\Delta L_i^n} H_i), \\
    Y_{n+1} &= e^{\Delta L^n}V^n_{n+1}.
\end{aligned}
\end{equation}
We will now look at some specific examples.

\subsection{Some examples of SRK Lawson schemes}\label{subsec:SRKLawsonExamples}
In the following{,} we will shortly discuss two SRK Lawson schemes for \ito\ SDEs, the Euler--Maruyama SL scheme and the Platen SL scheme, as well as the midpoint SL scheme for Stratonovich SDEs. In all cases{,} we assume that $\Delta W^n_0=h_n$ {and that for $m\geq1$,} $\Delta W^n_m$ is a suitable approximation to $W_m(t_{n+1})-W(t_n)$, {i.\,e., when we are interested in mean-square convergence of order $p$, then it needs to hold that $\E(\Delta W^n_m)=\gO(h_n^{p+1})$ and $\E[(\Delta W^n_m-(W_m(t_{n+1})-W(t_n)))^2]=\gO(h_n^{2p+1})$, while only the first $3$ moments of $\Delta W^n_m$ need to coincide with the ones of $W_m(t_{n+1})-W(t_n)$ if we are interested in weak convergence of order 1}.

\subsubsection{Euler--Maruyama SL scheme} The Euler--Maruyama scheme has the SRK coefficients $s=1$, $Z^{m,n}_{11}=0$, and $z_1^{m,n}=\Delta W^n_m$ and is mean square convergent of order 0.5 and weakly convergent of order 1. Following \cref{equ:SRKL} the corresponding Euler--Maruyama SL scheme is given by
\begin{equation} \label{equ:driftEuler}
\begin{multlined}
        Y_{n+1} = e^{\Delta L^n} Y_n + e^{\Delta L^n} \sum_{m=0}^M   \gt_m(t_n,Y_n)\Delta W^n_m.
\end{multlined}
\end{equation}
This scheme is, depending on how much of the linear diffusion is included in the exponential operator, also known as EI$0$, HomEI$0$ or Lawson Euler scheme, see e.g.\ \cite{erdogan19anc,komori14ase}.

\subsubsection{Platen SL scheme} The Platen scheme for $M=1$ \cite[Chapter 11.1]{kloeden99nso} has the SRK coefficients $s=2$, $z_1^{0,n}=h_n$, $z_1^{1,n}=\Delta W^n - \frac{1}{2\sqrt{h_n}}((\Delta W^n)^2-h_n)$, $z_2^{0,n}=0$, $z_2^{1,n}=\frac{1}{2\sqrt{h_n}}((\Delta W^n)^2-h_n)$, $Z^{m,n}_{1,j}=0$, $Z_{2,1}^{0,n}=h_n$, $Z_{2,1}^{1,n}=\sqrt{h_n}$, $Z^{m,n}_{2,2}=0$ and is mean square convergent of order 1. According to \cref{equ:SRKL} the resulting Platen SL scheme is given by
\begin{equation} \label{equ:driftPlaten}
\begin{aligned}
H_2 &= Y_n + \gt_0(t_n,Y_n)h_n + \gt_1(t_n,Y_n) \sqrt{h_n}, \\
    V^n_{n+1} &= Y_n + \gt_0(t_n,Y_n)h_n + \gt_1(t_n,Y_n) \Delta W^n + \frac{(\Delta W^n)^2 - h_n}{2\sqrt{h_n}}\\ &\qquad\times \Big[e^{-(A_0-\gs A^2_1)h_n -A_1 \sqrt{h_n}}\gt_1(t_n+h_n,e^{(A_0-\gs A^2_1)h_n +A_1 \sqrt{h_n}}H_2) -\gt_1(t_n,Y_n) \Big], \\
    Y_{n+1} &= e^{\Delta L^n}V^n_{n+1}.
\end{aligned}
\end{equation}
\subsubsection{Midpoint SL scheme}
The SRK coefficients for the stochastic implicit midpoint rule \cite[(2.39)]{milstein02nmf} are given by $s=1$, $ Z_{11}^{m,n} = 1/2 \Delta W_m^n$,
$z_{1}^{m,n} = \Delta W_m^n$. For commutative noise, the midpoint rule is mean square convergent of order 1, otherwise of order 0.5.
Applying \cref{equ:SRKL} we obtain
   \begin{equation*}
    \begin{aligned}
            H_1 &= Y_n + \sum_{m=0}^M \frac{1}{2} e^{-\frac12\Delta L^n}\gt_m \left(t_n+\frac{h_n}2,e^{\frac12\Delta L^n}H_1\right)\Delta W^n_m, \\
            V^n_{n+1} &= Y_n + \sum_{m=0}^Me^{-\frac12\Delta L^n}\gt_m \left(t_n+\frac{h_n}2,e^{\frac12\Delta L^n}H_1\right)\Delta W^n_m, \\
            Y_{n+1} &= e^{\Delta L^n}V^n_{n+1},
    \end{aligned}
    \end{equation*}
   which{,} by using that $H_1 = \frac{1}{2}(V^n_{n+1}+Y_n)${,} can be rewritten as
\begin{equation}\label{equ:fullMidpoint}
    \begin{aligned}
        Y_{n+1} &= e^{\Delta L^n} Y_n + \sum_{m=0}^M e^{\frac12 \Delta L^n}\gt_m \left(t_n+\frac{h_n}2,\frac{e^{\frac12 \Delta L^n}Y_n + e^{-\frac12 \Delta L^n}Y_{n+1}}{2} \right) \Delta W^n_m.
    \end{aligned}
\end{equation}

\section{Convergence of SRK Lawson schemes} \label{sec:SRK}
In this section, we will prove that the class of SRK Lawson schemes that we just constructed has, under some conditions, the same order of convergence as the underlying SRK scheme, thus removing the need for individual convergence proofs for the individual methods. Afterwards we will give some numerical examples. To simplify the presentation, from now on we will restrict to equidistant step sizes $h_n=h{=\frac{T-t_0}N}$, $n=0,\dots,N-1$.

\subsection{Strong and weak convergence} \label{sec:ProofConvergence}
To prove that the SRK Lawson schemes, under some conditions, inherit the consistency and convergence of the underlying SRK scheme, we first introduce global Lawson schemes in \cref{alg:GlobalLawson}.
\begin{algorithm}[ht]
\caption{Global stochastic Lawson scheme}\label{alg:GlobalLawson}
\begin{enumerate}[(a)]
  \item Apply a one-step method to get approximations $V^0_{n+1}$ of the exact solution $V^0(t_{n+1})$ of
\begin{equation}\label{equ:TransformedSDEGlob}
\dV^0(t) = \sum_{m=0}^M \gh^0_m(W(t),V^0(t))\star\dW_m(t), \quad V^0(t_0)=x_0
\end{equation}
for $n=0,\dots,N-1$.
\item For $n=0,\dots,N$, define
\begin{equation}\label{equ:TransformationYn}
    Y_{n}=e^{\bar{L}^0_n}V^0_{n}
\end{equation}
where $\bar{L}^0_n
{= \left(A_0-\gs\sum_{m=1}^M A_m^2\right)(t_n-t_0) + \sum_{m=1}^M A_m(W_m^n - W_m^0)
{\in\R^{d\times d}}}$ {with $W_m^n$ being the approximation of $W_m(t_n)$ induced by the one-step method used in step (i)}.
\end{enumerate}
\end{algorithm}
Even though \cref{alg:GlobalLawson} looks similar to \cref{alg:LocalLawson}, there are significant differences: In \cref{alg:LocalLawson} we integrate from $t_{n-1}$ to $t_n$, transform back and then define a new SDE for $V^n$. In contrast, in \cref{alg:GlobalLawson}, we integrate a single SDE for $V^0(t)$ from $t_0$ to $t_n$ before transforming back.

For \cref{alg:GlobalLawson}, we can prove that the global Lawson scheme inherits the strong convergence of the underlying one-step method:
\begin{lemma}[Strong convergence of \cref{alg:GlobalLawson}]\label{thm:convergence}
    Let \cref{ass:commute} hold and let $V^0_n$ be the numerical approximation of \cref{equ:TransformedSDEGlob} by some one-step method of mean square  order $p$, i.e.\  there exists a $c\in \mathbb{R}$ such that for all $N\in\N$ and all $n\in\{0,1,\dots,N\}$ it holds that $\sqrt{{\E}( \|V^0_n - V^0(t_n)\|_2^2)} \leq ch^p$. Moreover{,} assume that $\bar{L}^0_n=L^0(t_n)${, i.\,e., that $W^n=W(t_n)$, }and that $X$ is solution of SDE \eqref{equ:sde}.
    Then there exists a $\tilde{c} \in \mathbb{R}$ such that the numerical approximation $Y_n = e^{L^0(t_n)}V_n^0$ satisfies  for all $N\in\N$ and all $n\in\{0,1,\dots,N\}$
    \begin{equation}
        {\E} \|Y_n - X(t_n)\|_2 \leq \tilde{c}h^p.
    \end{equation}
\end{lemma}
{Here, $\|\cdot\|_2$ denotes the Euclidean vector norm.}
\begin{proof}
The strong convergence of \cref{alg:GlobalLawson} follows from $e^{L^0(t_n)}$ being mean square bounded \cite[Lemma 5.1]{erdogan19anc}, thus
\begin{align*}
    {\E} \|Y_n - X(t_n)\|_2 &= {\E} \|e^{L^0(t_n)}V^0_n - e^{L^0(t_n)}V^0(t_n)\|_2 \\&\leq c_1 \sqrt{{\E}( \|V^0_n - V^0(t_n)\|_2^2)} \\ &\leq c_1c h^p,
\end{align*}
{where we also used \cref{equ:TransformationYn}, \cref{equ:TransformedVar} (with $n=0$), and the Cauchy--Schwarz inequality.}
\end{proof}
A function $g:\R^n\to\R$ is called polynomially bounded if  there exist constants $K$ and $\kappa\geq0$ such that for all $x\in\R^n$ it holds that
\[
|g(x)|\leq K(1+\|x\|_2^\kappa).
\]
In the following we denote by $C^{2(\tilde{p}+1)}_P(\R^n,\R)$ the class of functions $g\in C^{2(\tilde{p}+1)}(\R^n,\R)$ for which $g$ and all its partial derivatives of order up to $2(\tilde{p}+1)$, inclusively, are polynomially bounded.
With this definition in place, we can also prove that the global Lawson scheme inherits the weak order of convergence of the underlying one-step method:
\begin{lemma}[Weak convergence of \cref{alg:GlobalLawson}]\label{thm:WeakConvergence}
Let \cref{ass:commute} hold and assume that $X$ is solution of SDE \eqref{equ:sde}. Moreover, assume that
\begin{enumerate}[(a)]
\item \label[item]{it:Amskewsymmetric}$A_m$ for $m> 0$ are skew-symmetric,
\item \label[item]{it:approxautonomous}{$W^n$ and $V^0_n$} are the numerical approximations of {$W(t_n)$ and} \cref{equ:TransformedSDEGlob} obtained by applying some one-step method to the autonomous system \cref{equ:autonomousSDE}, {and we assume especially that $W^n_0=W_0(t_n)=t_n$},
\item \label[item]{it:weakapproximation}this approximation is of weak order $\tilde{p}$, i.e.\ for all $g\in C^{2(\tilde{p}+1)}_P(\R^{M+1}\times\R^{d},\R)$ there exists a $c\in \mathbb{R}$ such that for all $N\in\N$ and all $n\in\{0,1,\dots,N\}$ it holds that $|{\E}(g(W^n,V^0_n) - g(W(t_n),V^0(t_n))| \leq ch^{\tilde{p}}$.
\end{enumerate}
    Then there exists for all $f\in C^{2(\tilde{p}+1)}_P(\R^d,\R)$ a $\tilde{c} \in \mathbb{R}$ such that the numerical approximation $Y_n = e^{\bar{L}^0_n}V_n^0$ satisfies for all $N\in\N$ and all $n\in\{0,1,\dots,N\}$
    \begin{equation}
        |{\E} f(Y_n) - f(X(t_n))|\leq \tilde{c}h^{\tilde{p}}.
    \end{equation}
\end{lemma}
\begin{proof}
For $s\in I= [t_0,T]$ {and $x\in\R^M$, $y\in\R^d$} let
\[
\hs_s({x,y})=f(e^{(A_0 + \gs\sum_{m=1}^M A_m)(s-t_0)}e^{\sum_{m=1}^M A_m ({x_m}-W^0_m)}{y}).
\]
Then it follows by \cref{it:approxautonomous} that
\[
f(Y_n)=\hs_{t_n}(W^n,V^0_n),\qquad f(X(t_n))=\hs_{t_n}(W(t_n),V^0(t_n)).
\]
To prove the assertion it is therefore by \cref{it:weakapproximation} enough to show that $\hs_s$ and all its derivatives up to order $2(\tilde{p}+1)$ satisfy an estimate of the form
\[
\|\hs_s({x,y})\|_2\leq K(1+\|{x}\|_2^{2\kappa}+\|{y}\|_2^{2\kappa})
\]
uniformly for all $s\in I$ (see \cite[Chapter 8]{milstein95nio}),
which follows from $f\in C_P^{2(\tilde{p}+1)}$ and ${\|e^{\sum_{m=1}^M A_m (x_m-W^0_m)}\|_2}=1$ due to \cref{it:Amskewsymmetric} {and the exponential of a skew symmetric matrix having Euclidean norm 1}.
\end{proof}
Using a {DSL} scheme trivially satisfies \cref{it:Amskewsymmetric,ass:commute}, and thus any {DSL} scheme immediately inherits the weak convergence of the underlying one-step method.

{
  We next demonstrate that for a given underlying SRK scheme, the local and the global Lawson scheme (\cref{alg:LocalLawson,alg:GlobalLawson}) result in
  the same output. Based on the autonomous form  \cref{equ:autonomousSDE}, an SRK
method applied to the autonomous version of the globally transformed equation \cref{equ:TransformedSDEGlob} can be written as}

\begin{equation}\label{equ:SRKglobal}
\begin{aligned}
    H^0_i &= V^0_n + \sum_{j=1}^s \sum_{m=0}^M Z_{ij}^{m{,n}} e^{-L^{0,n}_j}\gt_m(t_n+c^{n,j}_{{0}},e^{L^{0,n}_j} H^0_j), \\
    V^0_{n+1} &= V^0_n + \sum_{i=1}^s \sum_{m=0}^M z_i^{m{,n}} e^{-L^{0,n}_i}\gt_m({t_n}+c^{n,i}_{{0}}, e^{L^{0,n}_i}H^0_i)
\end{aligned}
\end{equation}
with \[L^{0,n}_i= \left(A_0 - \gs \sum_{m=1}^M A_m\right)(t_n+{c^{n,i}_{0}}-t_0) + \sum_{m=1}^M A_m(W_m^n+c^{n,i}_m{-W_m^0})\]
{and with $\gt_m(t,x)$ defined in \cref{lem:transformation}.
Using this formulation, we get the following result:}

\begin{lemma}\label{lem:comparison}
       \Cref{alg:LocalLawson} {with an arbitrary SRK method \cref{equ:SRK} applied to \cref{equ:TransformedSDELocal}} and \cref{alg:GlobalLawson} {with the same underlying SRK method \cref{equ:SRKglobal} applied to \cref{equ:TransformedSDEGlob}} give the same sequence of approximation points $Y_n$.
\end{lemma}
\begin{proof}
    We prove this by induction. Let $Y_n$ denote the numerical solution by \cref{alg:LocalLawson}. Similarly, let $e^{\bar{L}^0_n}V^0_{n}$ be the numerical approximation obtained by \cref{alg:GlobalLawson}. It holds that $Y_0=e^{\bar{L}^0_0}V^0_{0}=x_0$. We assume now that it {holds} that $e^{\bar{L}^0_n}V^0_{n} = Y_n$, and prove that this implies that $e^{\bar{L}^0_{n+1}}V^0_{n+1} = Y_{n+1}$:

    For this, we consider a single update from $t_n$ to $t_{n+1}$ using the SRK method {\eqref{equ:SRKglobal}} {for \cref{alg:GlobalLawson}}.
Defining $H_i=e^{\bar{L}^0_n}H_i^0$ and multiplying $V^0_{n+1}$ by $e^{\bar{L}^0_{n+1}}$ it holds then
    \begin{align*}
        H_i &= e^{\bar{L}^0_n} V^0_n + \sum_{j=1}^{s} \sum_{m=0}^{M} Z_{ij}^{m{,n}} e^{\bar{L}^0_n - L^{0,n}_j}\gt_m(t_n+c_0^{n,j},e^{L^{0,n}_j-\bar{L}^0_n}H_j),\\
      e^{\bar{L}^0_{n+1}} V^0_{n+1}
        &= e^{\bar{L}^0_{n+1}} V^0_n + \sum_{i=1}^{s}\sum_{m=0}^M z_i^{m{,n}} e^{\bar{L}^0_{n+1}-L^{0,n}_i}\gt_m(t_n+{c^{n,i}_{0}}, e^{L^{0,n}_i-\bar{L}^0_n} H_i).
    \end{align*}

    \noindent Using the induction hypothesis, $Y_n = e^{\bar{L}^0_n}V^0_n$, and $L_i^{0,n} - \bar{L}^0_n = \Delta L_i^n${,} we obtain
    \begin{align*}
    H_i &= Y_n + \sum_{j=1}^{s} \sum_{m=0}^{M} Z_{ij}^{m{,n}} e^{-\Delta L_j^n}\gt_m(t_n+{c_0^{n,j}},e^{\Delta L_j^n}H_j),\\
        e^{\bar{L}^0_{n+1}} V^0_{n+1}&= e^{\bar{L}^0_{n+1}-\bar{L}^0_n} Y_n + \sum_{i=1}^{s}\sum_{m=0}^M z_i^{m{,n}} e^{\bar{L}^0_{n+1}-\bar{L}^0_n - \Delta L_i^n}\gt_m(t_n+c^{n,i}_{{0}}, e^{\Delta L_i^n} H_i).
    \end{align*}

    \noindent Finally by using $\Delta L^n= \bar{L}^0_{n+1}-\bar{L}^0_n$ we obtain
    \begin{align*}
        H_i &=  Y_n + \sum_{j=1}^{s} \sum_{m=0}^{M} Z_{ij}^{m{,n}} e^{- \Delta L^n_j}\gt_m(t_n+{c_0^{n,j}},e^{\Delta L^n_j}H_j), \\
        e^{\bar{L}^0_{n+1}} V^0_{n+1}&= e^{\Delta L^n} \left[ Y_n + \sum_{i=1}^{s}\sum_{m=0}^M z_i^{m{,n}} e^{-\Delta L^n_i}\gt_m(t_n+{c^{n,i}_{0}},e^{\Delta L^n_i} H_i) \right],
    \end{align*}
    which is identical to {\cref{alg:LocalLawson} given in} \eqref{equ:SRKL}, and consequently it holds that $Y_{n+1}=e^{\bar{L}^0_{n+1}} V^0_{n+1}$.
\end{proof}

\noindent It follows that the (local) Lawson scheme converges:
\begin{theorem}[Strong convergence of \cref{alg:LocalLawson}]\label{thm:convergenceLocalScheme}
Let \cref{ass:commute} hold, let $X$ be the solution of SDE \eqref{equ:sde} and $Y_n$ be the result of applying \cref{alg:LocalLawson} with an underlying SRK method of mean square  order $p$. {Assume further that  $\Delta W_m^n=W_m(t_{n+1})-W_m(t_n)$ in \cref{eq:DeltaLndef}.} Then there exists a $\tilde{c} \in \mathbb{R}$ such that for all $N\in\N$ and all $n\in\{0,1,\dots,N\}$ it holds {that}
    \begin{equation}
        {\E}\|Y_n - X(t_n)\|_2 \leq \tilde{c}h^p.
    \end{equation}
\end{theorem}
\begin{proof}
    This is proven by combining \cref{lem:comparison,thm:convergence} (choosing $\bar{L}^0_{n} = \sum_{i=0}^{n-1} \Delta L^i$).
\end{proof}
\begin{theorem}[Weak convergence of \cref{alg:LocalLawson}]
Let \cref{ass:commute} hold and assume that $X$ is {the} solution of SDE \eqref{equ:sde}. Moreover, assume that
\begin{enumerate}[(a)]
\item $A_m$ for $m> 0$ are skew-symmetric,
\item {$W^n$ and} $V^{{n}}_{{n+1}}$ are the numerical approximations of {$W(t_n)$ and } {\cref{equ:TransformedSDELocal} obtained by the SRK method \cref{equ:SRKL}},
\item this approximation is of weak order $\tilde{p}$, i.e.\ for all $g\in C^{2(\tilde{p}+1)}_P(\R^{M+1}\times\R^{d},\R)$ there exists a $c\in \mathbb{R}$ such that for all $N\in\N$ and all $n\in\{0,1,\dots,N\}$ it holds that $|{\E}(g(W^n,V^0_n) - g(W(t_n),V^0(t_n))| \leq ch^{\tilde{p}}$,
\item $W^n_0=W_0(t_n)=t_n$.
\end{enumerate}

    Then for all $f\in C^{2(\tilde{p}+1)}_P{(\R^{d},\R)}$ {there exists some constant $c_f > 0$ such that} the numerical approximation $Y_{n+1}=e^{\Delta L^n}V^n_{n+1}$ satisfies {for all $N\in\N$ and all $n\in\{0,1,\dots,N\}$}
    \begin{equation}
         | {\E} f(Y_{{n}}) - {\E}f(X(t_{{n}}))| \leq c_{{f}} h^{\tilde{p}}.
    \end{equation}
    \end{theorem}
\begin{proof}
    This is proven by combining \cref{lem:comparison,thm:WeakConvergence} (and choosing $\bar{L}^0_{n+1} = \sum_{i=0}^{n} \Delta L^i$).
\end{proof}
We want to emphasize that for Lipschitz continuous $f${,} the weak convergence order {is at least as big as} the strong order of convergence. So in the case where not all $\{A_m\}_{m>0}$ are skew-symmetric, we still have weak convergence of  {at least} the same order as the strong order. However, this is not necessarily of the same order as the weak order of the underlying SRK scheme.

For the convergence of exponential schemes that do not originate from a Lawson type transformation{,} we refer to \cite{arara19sbs}. The convergence of {SRK} schemes is rather well studied and we refer to e.\,g.\ \cite{burrage96hso,roessler07sor,debrabant08bao,roessler09sor} for the analysis of the mean square convergence order and \cite{komori97rta,burrage00oco,roessler06rta,debrabant08bao,roessler10saw} for the analysis of the weak convergence order.

In the following, we present some numerical examples showing the expected order of convergence.

\subsection{Numerical example}\label{sec:ConvNumExample}
In this section, we numerically show the strong ($p$) and weak ($\tilde{p}$) order of convergence of the {following} schemes:
\begin{itemize}
    \item {T}he Euler--Maruyama DSL scheme (\enquote{EM DSL}), \cref{equ:driftEuler} with $A_m=0$ for $m>0$, with expected strong and weak convergence orders $p=0.5$ and $\tilde{p}=1$.
    \item {T}he Platen DSL scheme (\enquote{Platen DSL}), {defined in \cref{equ:driftPlaten} for $M=1$}, {with} $A_m=0$ for $m>0$, with expected strong and weak convergence orders $p=1$ and $\tilde{p}=1$.
    \item {T}he Midpoint FSL scheme (\enquote{Midpoint FSL}), \cref{equ:fullMidpoint}, with expected strong and weak convergence orders $p=1$ and $\tilde{p}=1$. In the numerical experiments, the {solution of the} implicit equation will be  {approximated with one single Newton step}, evaluating the Jacobian once per time step{, which has been proved sufficient to maintain the correct order, see \cite{debrabant08bao}. For the below example, we also tested numerically that the results are not influenced significantly by doing more iterations.}
    \item {T}he Platen strong order $1.5$ DSL scheme (\enquote{Platen 1.5 DSL}), defined in \cref{sec:Platen15} for $M=1$, with $A_m=0$ for $m>0$, with expected strong convergence order $p=1.5$. For this scheme, we only demonstrate numerically the order of strong convergence.
    \item {T}he Platen weak order $2.0$ DSL scheme (\enquote{Platen 2.0 DSL}), defined in \cref{sec:Platen2} for $M=1$, with $A_m=0$ for $m>0$, with expected weak convergence order $p=2.0$. For this scheme, we only demonstrate numerically the order of weak convergence.
\end{itemize}
{For our numerical investigations we will use the mildly stiff {\ito} SDE}
\begin{equation}\label{eq:testSDEstrongconvergence}
\begin{multlined}
        \dX =\left[\begin{pmatrix}
        -\lambda & 0 \\ 0 & -\lambda
    \end{pmatrix} X + \begin{pmatrix}
        0 & U_0(X) \\ -U_0(X) & 0
    \end{pmatrix}{X}\right] \dt {+} \begin{pmatrix}
        0 & 0.2 \\ -0.2 & 0
    \end{pmatrix} X \dW,
\end{multlined}
\end{equation}
with $X(0)=(1,0)^{\top}$, $U_0(X) = \frac{1}{5}(X_1+X_2)^5${, $t\in[0,1]$ and $\lambda\in\{1,5\}$}. This corresponds to a non-linear oscillator {\cite[Eq.\ (18)]{cohen12otn}} {perturbed by} a linear attractor. {The Midpoint FSL scheme will be applied to the corresponding Stratonovich version.} {By \ito's formula it follows that $\dmath\|X(t)\|_2^2=(-2\lambda+0.04)\|X(t)\|_2^2\dt$ and thus {$\|X(t)\|_2=e^{(-\lambda+0.02)t}\to 0$ for $t\to \infty$ if $\lambda>0.2$}.}
{In all cases we choose $A_0=-\bigl(\begin{smallmatrix}
\lambda & 0 \\ 0 & \lambda
\end{smallmatrix}\bigr)$ and run the numerical experiments in MATLAB version R2020a with 20 single-threaded workers on a 4×6-core Xeon 2.67 GHz Linux (Ubuntu 18.04) machine with 256 GB memory.} {The MATLAB source code for all numerical experiments in this section can be found at \cite{debrabant20mcr}.}
\subsubsection{Strong {convergence}}
We simulate {40 batches of each 25 paths for SDE} \eqref{eq:testSDEstrongconvergence} {and} compare the results to a numerical solution obtained by the {Platen 1.5 DSL} scheme with step size $h=2^{-{24}}$.

{For step sizes $h\in\{2^{-13},2^{-14},2^{-15}, 2^{-16}, 2^{-17}, 2^{-18}\}$ we} report the average errors for $\lambda=1$ and $\lambda=5$ in \cref{fig:convergence,fig:convergence5}. The {$95\%$}-confidence intervals have been calculated and span in all cases less than {$\pm 6.1\%$} of the corresponding error value.

\pgfplotstableread[col sep=comma]{Data/Convergence/MildlyStiffItoSDE_OuterIter=1_Miter=40_Mbatch=25_Nh=24_E_EDSL_P_PDSL_M_MFSL_P15_P15DSL_Param=1_red.csv}\ConvStrongLambdaOne
\pgfplotstableread[col sep=comma]{Data/Convergence/MildlyStiffItoSDE_OuterIter=1_Miter=40_Mbatch=25_Nh=24_E_EDSL_P_PDSL_M_MFSL_P15_P15DSL_Param=1.csv}\TimingsStrongLambdaOne
\pgfplotstableread[col sep=comma]{Data/Convergence/MildlyStiffItoSDE_OuterIter=1_Miter=40_Mbatch=25_Nh=24_E_EDSL_P_PDSL_M_MFSL_P15_P15DSL_Param=5_red.csv}\ConvStrongLambdaFive
\pgfplotstableread[col sep=comma]{Data/Convergence/MildlyStiffItoSDE_OuterIter=1_Miter=40_Mbatch=25_Nh=24_E_EDSL_P_PDSL_M_MFSL_P15_P15DSL_Param=5.csv}\TimingsStrongLambdaFive

\begin{figure}[htp!]

\begin{subfigure}{\textwidth}
\begin{tikzpicture}
    \begin{axis}[
        name=plot1,
        width=0.9\textwidth,
        height=0.2\textheight,
        axis x line=bottom,
        xmode = log,
        log basis x=2,
        axis y line=left,
        ymode = log,
        log basis y=2,
        xlabel=$h$,
        ylabel={${\hat{\E}}\|Y_N - X(T_N)\|_2$},
        domain=0:1,
        legend style={at={(0.5,1.5)},anchor=north},
        legend columns=4]
        \addplot[name path,only marks,blue,mark=*] table[x ={h}, y={eEDSL}]{\ConvStrongLambdaOne};
        \addplot[mark=none, blue,forget plot] table[x={h}, y={create col/linear regression={y=eEDSL}}]{\ConvStrongLambdaOne};
        \xdef\slopeEDSL{\pgfplotstableregressiona}
        \addlegendentry{EM DSL, $p=\pgfmathprintnumber{\slopeEDSL}$}

        \addplot[name path,only marks,red,mark=*] table[x ={h}, y={ePDSL}]{\ConvStrongLambdaOne};
        \addplot[mark=none, red,forget plot] table[x={h}, y={create col/linear regression={y=ePDSL}}]{\ConvStrongLambdaOne};
        \xdef\slopePDSL{\pgfplotstableregressiona}
        \addlegendentry{Platen DSL, $p=\pgfmathprintnumber{\slopePDSL}$}

        \addplot[name path,only marks,green,mark=*] table[x ={h}, y={eMFSL}]{\ConvStrongLambdaOne};
        \addplot[mark=none, green,forget plot] table[x={h}, y={create col/linear regression={y=eMFSL}}]{\ConvStrongLambdaOne};
        \xdef\slopeMFSL{\pgfplotstableregressiona}
        \addlegendentry{Midpoint FSL, $p=\pgfmathprintnumber{\slopeMFSL}$}

        \addplot[name path,only marks,black,mark=*] table[x ={h}, y={eP15DSL}]{\ConvStrongLambdaOne};
        \addplot[mark=none, black,forget plot] table[x={h}, y={create col/linear regression={y=eP15DSL}}]{\ConvStrongLambdaOne};
        \xdef\slopePlatenDSL{\pgfplotstableregressiona}
        \addlegendentry{Platen 1.5 DSL, $p=\pgfmathprintnumber{\slopePlatenDSL}$}

        %%%%%%%%%%%%%%%%%%%%%%%%%%%%%%%%

        \addplot[name path,only marks,blue,mark=diamond*] table[x ={h}, y={eE}]{\ConvStrongLambdaOne};
        \addplot[mark=none, blue,forget plot, dashed] table[x={h}, y={create col/linear regression={y=eE}}]{\ConvStrongLambdaOne};
        \xdef\slopeE{\pgfplotstableregressiona}
        \addlegendentry{EM, $p=\pgfmathprintnumber{\slopeE}$}

        \addplot[name path,only marks,red,mark=diamond*] table[x ={h}, y={eP}]{\ConvStrongLambdaOne};
        \addplot[mark=none, red,forget plot,dashed] table[x={h}, y={create col/linear regression={y=eP}}]{\ConvStrongLambdaOne};
        \xdef\slopeP{\pgfplotstableregressiona}
        \addlegendentry{Platen, $p=\pgfmathprintnumber{\slopeP}$}

        \addplot[name path,only marks,green,mark=diamond*] table[x ={h}, y={eM}]{\ConvStrongLambdaOne};
        \addplot[mark=none, green,forget plot,dashed] table[x={h}, y={create col/linear regression={y=eM}}]{\ConvStrongLambdaOne};
        \xdef\slopeM{\pgfplotstableregressiona}
        \addlegendentry{Midpoint, $p=\pgfmathprintnumber{\slopeM}$}

        \addplot[name path,only marks,black,mark=diamond*] table[x ={h}, y={eP15}]{\ConvStrongLambdaOne};
        \addplot[mark=none, black,forget plot,dashed] table[x={h}, y={create col/linear regression={y=eP15}}]{\ConvStrongLambdaOne};
        \xdef\slopePlaten{\pgfplotstableregressiona}
        \addlegendentry{Platen 1.5, $p=\pgfmathprintnumber{\slopePlaten}$}
    \end{axis}
\end{tikzpicture}
    \caption{{Strong error vs.\ steps size} {for $\lambda=1$}}\label{fig:convergence}
\end{subfigure}
\vspace{10pt}
\begin{subfigure}{\textwidth}
\begin{tikzpicture}
    \begin{axis}[
        name=plot1var2,
        width=0.9\textwidth,
        height=0.2\textheight,
        axis x line=bottom,
        xmode = log,
        log basis x=2,
        axis y line=left,
        ymode = log,
        log basis y=2,
        xlabel=$h$,
        ylabel={$\hat{\E}\|Y_N - X(T_N)\|_2$},
        domain=0:1,
        legend style={at={(0.5,1.5)},anchor=north},
        legend columns=4]
        \addplot[name path,only marks,blue,mark=*] table[x ={h}, y={eEDSL}]{\ConvStrongLambdaFive};
        \addplot[mark=none, blue,forget plot] table[x={h}, y={create col/linear regression={y=eEDSL}}]{\ConvStrongLambdaFive};
        \xdef\slopeEDSL{\pgfplotstableregressiona}
        \addlegendentry{EM DSL, $p=\pgfmathprintnumber{\slopeEDSL}$}

        \addplot[name path,only marks,red,mark=*] table[x ={h}, y={ePDSL}]{\ConvStrongLambdaFive};
        \addplot[mark=none, red,forget plot] table[x={h}, y={create col/linear regression={y=ePDSL}}]{\ConvStrongLambdaFive};
        \xdef\slopePDSL{\pgfplotstableregressiona}
        \addlegendentry{Platen DSL, $p=\pgfmathprintnumber{\slopePDSL}$}

        \addplot[name path,only marks,green,mark=*] table[x ={h}, y={eMFSL}]{\ConvStrongLambdaFive};
        \addplot[mark=none, green,forget plot] table[x={h}, y={create col/linear regression={y=eMFSL}}]{\ConvStrongLambdaFive};
        \xdef\slopeMFSL{\pgfplotstableregressiona}
        \addlegendentry{Midpoint FSL, $p=\pgfmathprintnumber{\slopeMFSL}$}

        \addplot[name path,only marks,black,mark=*] table[x ={h}, y={eP15DSL}]{\ConvStrongLambdaFive};
        \addplot[mark=none, black,forget plot] table[x={h}, y={create col/linear regression={y=eP15DSL}}]{\ConvStrongLambdaFive};
        \xdef\slopePlatenDSL{\pgfplotstableregressiona}
        \addlegendentry{Platen 1.5 DSL, $p=\pgfmathprintnumber{\slopePlatenDSL}$}

        %%%%%%%%%%%%%%%%%%%%%%%%%%%%%%%%

        \addplot[name path,only marks,blue,mark=diamond*] table[x ={h}, y={eE}]{\ConvStrongLambdaFive};
        \addplot[mark=none, blue,forget plot, dashed] table[x={h}, y={create col/linear regression={y=eE}}]{\ConvStrongLambdaFive};
        \xdef\slopeE{\pgfplotstableregressiona}
        \addlegendentry{EM, $p=\pgfmathprintnumber{\slopeE}$}

        \addplot[name path,only marks,red,mark=diamond*] table[x ={h}, y={eP}]{\ConvStrongLambdaFive};
        \addplot[mark=none, red,forget plot,dashed] table[x={h}, y={create col/linear regression={y=eP}}]{\ConvStrongLambdaFive};
        \xdef\slopeP{\pgfplotstableregressiona}
        \addlegendentry{Platen, $p=\pgfmathprintnumber{\slopeP}$}

        \addplot[name path,only marks,green,mark=diamond*] table[x ={h}, y={eM}]{\ConvStrongLambdaFive};
        \addplot[mark=none, green,forget plot,dashed] table[x={h}, y={create col/linear regression={y=eM}}]{\ConvStrongLambdaFive};
        \xdef\slopeM{\pgfplotstableregressiona}
        \addlegendentry{Midpoint, $p=\pgfmathprintnumber{\slopeM}$}

        \addplot[name path,only marks,black,mark=diamond*] table[x ={h}, y={eP15}]{\ConvStrongLambdaFive};
        \addplot[mark=none, black,forget plot,dashed] table[x={h}, y={create col/linear regression={y=eP15}}]{\ConvStrongLambdaFive};
        \xdef\slopePlaten{\pgfplotstableregressiona}
        \addlegendentry{Platen 1.5, $p=\pgfmathprintnumber{\slopePlaten}$}
    \end{axis}
\end{tikzpicture}
    \caption{{Strong error vs.\ steps size for $\lambda=5$}}\label{fig:convergence5}
\end{subfigure}
\vspace{10pt}
\begin{subfigure}{\textwidth}
\begin{tikzpicture}
    \newenvironment{customlegend}[1][]{%
        \begingroup
        \csname pgfplots@init@cleared@structures\endcsname
        \pgfplotsset{#1}%
    }{%
        % draws the legend:
        \csname pgfplots@createlegend\endcsname
        \endgroup
    }%
\def\addlegendimage{\csname pgfplots@addlegendimage\endcsname}

\begin{customlegend}[legend columns=4,legend style={align=left,draw=none,column sep=2ex},legend entries={EM DSL, Platen DSL, Midpoint FSL, Platen $1.5$ DSL,EM, Platen, Midpoint, Platen $1.5$}]
        \addlegendimage{blue,mark=*}
        \addlegendimage{red,mark=*}
        \addlegendimage{green,mark=*}
        \addlegendimage{black,mark=*}
        \addlegendimage{blue,dashed,mark=diamond*}
        \addlegendimage{red,dashed,mark=diamond*}
        \addlegendimage{green,dashed,mark=diamond*}
        \addlegendimage{black,dashed,mark=diamond*}
\end{customlegend}
\end{tikzpicture}
\end{subfigure}

\begin{subfigure}{0.45\textwidth}
\xdef\basis{2}
\begin{tikzpicture}
    \begin{axis}[
        name=plot1,
        width=0.9\textwidth,
        height=0.29\textheight,
        axis x line=bottom,
        xmode = log,
        log basis x=\basis,
        axis y line=left,
        ymode = log,
        ymax = 1,
        ymin= 0.001,
        log basis y=\basis,
        title={},
        ylabel={Wall-clock time $[s]$},
        xlabel={${\hat{\E}}\|Y_N - X(T_N)\|_2$},
        legend pos=north east]
        \addplot[name path,blue,mark=*] table[x ={eEDSL}, y={tEDSL}]{\TimingsStrongLambdaOne};

        \addplot[name path,red,mark=*] table[x ={ePDSL}, y={tPDSL}]{\TimingsStrongLambdaOne};

        \addplot[name path,green,mark=*] table[x ={eMFSL}, y={tMFSL}]{\TimingsStrongLambdaOne};

        \addplot[name path,black,mark=*] table[x ={eP15DSL}, y={tP15DSL}]{\TimingsStrongLambdaOne};

        %%%%%%%%%%%%%%%%

        \addplot[name path,dashed,blue,mark=diamond*] table[x ={eE}, y={tE}]{\TimingsStrongLambdaOne};
        \addplot[name path,dashed,red,mark=diamond*] table[x ={eP}, y={tP}]{\TimingsStrongLambdaOne};
        \addplot[name path,dashed,green,mark=diamond*] table[x ={eM}, y={tM}]{\TimingsStrongLambdaOne};
        \addplot[name path,dashed,black,mark=diamond*] table[x ={eP15}, y={tP15}]{\TimingsStrongLambdaOne};
    \end{axis}
\end{tikzpicture}
    \caption{{Wall-clock time per batch of 25 paths vs.\ accuracy for $\lambda=1$}}\label{fig:efficiency1}
\end{subfigure}\hspace*{\fill}
\begin{subfigure}{0.45\textwidth}
\xdef\basis{2}
\begin{tikzpicture}
    \begin{axis}[
        name=plot1,
        width=0.9\textwidth,
        height=0.29\textheight,
        axis x line=bottom,
        xmode = log,
        log basis x=\basis,
        axis y line=left,
        ymode = log,
       ymax = 1,
       ymin= 0.001,
        log basis y=\basis,
        title={},
        ylabel={Wall-clock time $[s]$},
        xlabel={${\hat{\E}}\|Y_N - X(T_N)\|_2$},
        legend pos=north east]
        \addplot[name path,blue,mark=*] table[x ={eEDSL}, y={tEDSL}]{\TimingsStrongLambdaFive};

        \addplot[name path,red,mark=*] table[x ={ePDSL}, y={tPDSL}]{\TimingsStrongLambdaFive};

        \addplot[name path,green,mark=*] table[x ={eMFSL}, y={tMFSL}]{\TimingsStrongLambdaFive};

        \addplot[name path,black,mark=*] table[x ={eP15DSL}, y={tP15DSL}]{\TimingsStrongLambdaFive};

        %%%%%%%%%%%%%%%%

        \addplot[name path,dashed,blue,mark=diamond*] table[x ={eE}, y={tE}]{\TimingsStrongLambdaFive};

        \addplot[name path,dashed,red,mark=diamond*] table[x ={eP}, y={tP}]{\TimingsStrongLambdaFive};

        \addplot[name path,dashed,green,mark=diamond*] table[x ={eM}, y={tM}]{\TimingsStrongLambdaFive};

        \addplot[name path,dashed,black,mark=diamond*] table[x ={eP15}, y={tP15}]{\TimingsStrongLambdaFive};

    \end{axis}
\end{tikzpicture}
    \caption{{Wall-clock time per batch of 25 paths vs.\ accuracy for $\lambda=5$}}\label{fig:efficiency2}
\end{subfigure}\hspace*{\fill}

\caption{{Strong convergence results for SDE \cref{eq:testSDEstrongconvergence}}}
\end{figure}
In \cref{fig:convergence} we see that {for $\lambda=1$} the Euler--Maruyama DSL scheme has an approximate strong order $p=0.5${,} the Platen DSL and Midpoint FSL schemes have an approximate strong order $p=1.0${, and the Platen $1.5$ DSL scheme has an approximate strong order of $1.5$} as expected. {Of the two first order schemes, the} Midpoint FSL scheme has a smaller error than the Platen DSL, which is also expected as this scheme, {in addition to the drift term,} also incorporates the diffusion term into the exponential operator.

{In \cref{fig:convergence5} we observe that for $\lambda=5$ the errors of the SL schemes are
considerably smaller than the ones of their underlying methods. The order is as expected for all
methods {except the} EM method. {For the latter} the order 1 error term dominates the order 0.5 term for
the given step sizes, and {we observe that} the error more or less completely overlaps with the one
of Platen's method. We still notice that the error of EM DSL is smaller than the one of EM,
although the measured order is smaller. For step sizes less than $2^{-18}$, we {have observed} that the error difference between EM and EM DSL is insignificant, and both methods exhibit order 0.5.}

{In \cref{fig:efficiency1,fig:efficiency2} we depict the computational effort, measured as wall-clock time per batch of 25 paths, vs.\ the strong error {averaged over all batches}.
Between the two Euler-Maruyama methods {and the two Midpoint methods, respectively,} {a difference in efficiency can only be seen for $\lambda =5$, while} the DSL Platen methods are more efficient than their counterparts for both values of $\lambda$. The increased efficiency of the DSL-methods is more pronounced for $\lambda=5$.
{For larger step sizes and especially $\lambda=5$, similar to the above discussion it can be observed that the error of both the Euler-Maruyama and the Midpoint method is dominated by higher order terms.}
}
\subsubsection{Weak convergence}
{We simulate 2400 batches of each 15000 paths of SDE \eqref{eq:testSDEstrongconvergence} and compare the results to a numerical solution obtained by the {Platen 1.5 DSL} scheme with step size $h=2^{-14}$ as reference solution.}

{In \cref{fig:Weakconvergence,fig:Weakconvergence5}, we report for step sizes $h\in\{2^{-9},2^{-8},2^{-7}, 2^{-6}, 2^{-5}\}$
the error when approximating $\E(X_1^2)$ for $\lambda=1$ and $\lambda=5$ . The $95\%$-confidence intervals have been calculated and span in all cases less than {$6\%$} of the corresponding error value, except for the Platen $2$ DSL, where they span $10\%$, $14\%$, $33\%$, $68\%$, and $122\%$ for $\lambda=1$ and
$2\%$, $6\%$, $25\%$, $98\%$, and $359\%$ for $\lambda=5$
for step-sizes $h=2^{-5}$, $2^{-6}$, $2^{-7}$, $2^{-8}$, and $2^{-9}$ respectively, as well as the {EM} DSL scheme in the case of $\lambda=5$, where the confidence intervals span up to 27\% of the corresponding error value.}

\pgfplotstableread[col sep=comma]{Data/Convergence/WeakApprox_OuterIter=6_Miter=400_Mbatch=15000_Nh=14_E_EDSL_P_PDSL_Pw2_Pw2DSL_1_red.csv}\ConvWeak
\pgfplotstableread[col sep=comma]{Data/Convergence/WeakApprox_OuterIter=6_Miter=400_Mbatch=15000_Nh=14_E_EDSL_P_PDSL_Pw2_Pw2DSL_5_red.csv}\ConvWeakFive
\pgfplotstableread[col sep=comma]{Data/Convergence/WeakApprox_OuterIter=6_Miter=400_Mbatch=15000_Nh=14_E_EDSL_P_PDSL_Pw2_Pw2DSL_1.csv}\Timings
\pgfplotstableread[col sep=comma]{Data/Convergence/WeakApprox_OuterIter=6_Miter=400_Mbatch=15000_Nh=14_E_EDSL_P_PDSL_Pw2_Pw2DSL_5.csv}\TimingsFive

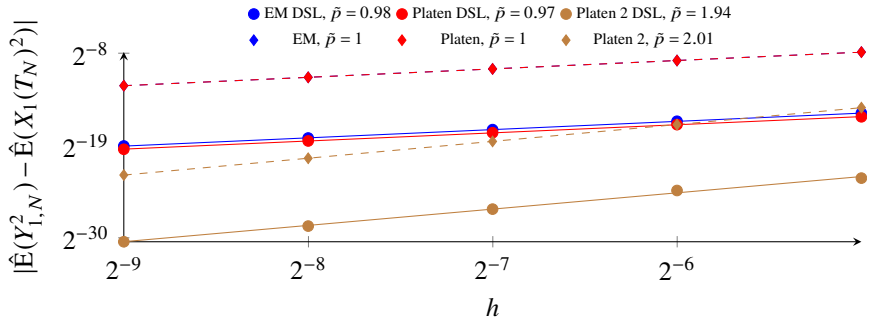
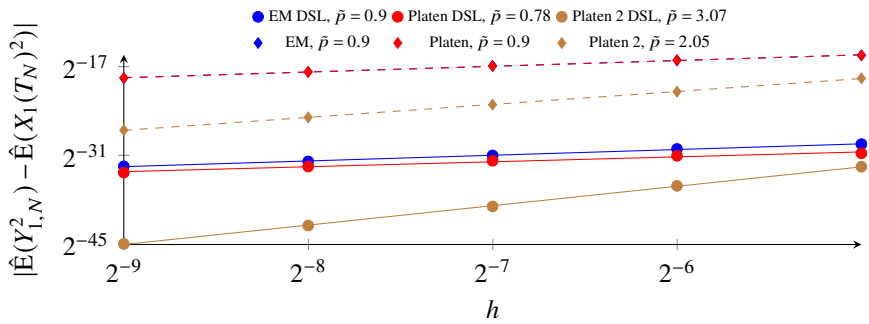
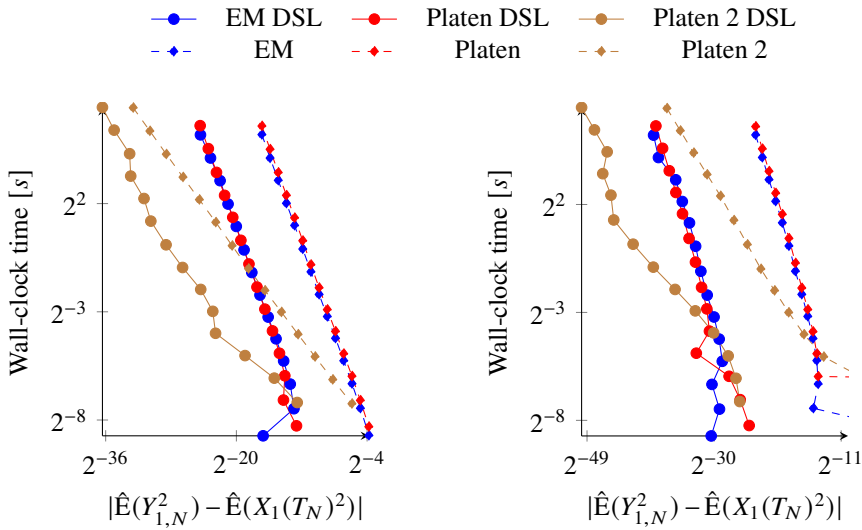
\begin{figure}[htp!]

\begin{subfigure}{\textwidth}
\begin{tikzpicture}
    \begin{axis}[
        name=plot2,
        anchor=above north east,
        width=0.9\textwidth,
        height=0.2\textheight,
        axis x line=bottom,
        xmode = log,
        log basis x=2,
        axis y line=left,
        ymode = log,
        log basis y=2,
        xlabel=$h$,
        ylabel={$| {\hat{\E}}(Y^2_{1,N}) - {\hat{\E}}( X_1(T_N)^2)|$},
        domain=0:1,
        legend style={at={(0.5,1.3)},anchor=north},
        legend columns=3]
        \pgfplotstablecreatecol[create col/expr={abs(\thisrow{eEDSL})}]{abseEDSL}\ConvWeak
        \addplot[name path,only marks,blue, mark=*] table[x ={h}, y={abseEDSL}]{\ConvWeak};
        \addplot[mark=none, blue,forget plot] table[x={h}, y={create col/linear regression={y=abseEDSL}}]{\ConvWeak};
        \xdef\slopeEDSL{\pgfplotstableregressiona}
        \addlegendentry{EM DSL, $\tilde{p}=\pgfmathprintnumber{\slopeEDSL}$}

        \pgfplotstablecreatecol[create col/expr={abs(\thisrow{ePDSL})}]{absePDSL}\ConvWeak
        \addplot[name path,only marks,red, mark=*] table[x ={h}, y={absePDSL}]{\ConvWeak};
        \addplot[mark=none, red,forget plot] table[x={h}, y={create col/linear regression={y=absePDSL}}]{\ConvWeak};
        \xdef\slopePDSL{\pgfplotstableregressiona}
        \addlegendentry{Platen DSL, $\tilde{p}=\pgfmathprintnumber{\slopePDSL}$}

        \pgfplotstablecreatecol[create col/expr={abs(\thisrow{ePw2DSL})}]{absePw2DSL}\ConvWeak
        \addplot[name path,only marks,brown, mark=*] table[x ={h}, y={absePw2DSL}]{\ConvWeak};
        \addplot[mark=none, brown,forget plot] table[x={h}, y={create col/linear regression={y=absePw2DSL}}]{\ConvWeak};
        \xdef\slopePwDSL{\pgfplotstableregressiona}
        \addlegendentry{Platen 2 DSL, $\tilde{p}=\pgfmathprintnumber{\slopePwDSL}$}

        \pgfplotstablecreatecol[create col/expr={abs(\thisrow{eE})}]{abseE}\ConvWeak
        \addplot[name path,only marks,blue,mark=diamond*] table[x ={h}, y={abseE}]{\ConvWeak};
        \addplot[mark=none, blue,forget plot,dashed] table[x={h}, y={create col/linear regression={y=abseE}}]{\ConvWeak};
        \xdef\slopeE{\pgfplotstableregressiona}
        \addlegendentry{EM, $\tilde{p}=\pgfmathprintnumber{\slopeE}$}

        \pgfplotstablecreatecol[create col/expr={abs(\thisrow{eP})}]{abseP}\ConvWeak
        \addplot[name path,only marks,red,mark=diamond*] table[x ={h}, y={abseP}]{\ConvWeak};
        \addplot[mark=none, red,forget plot,dashed] table[x={h}, y={create col/linear regression={y=abseP}}]{\ConvWeak};
        \xdef\slopeP{\pgfplotstableregressiona}
        \addlegendentry{Platen, $\tilde{p}=\pgfmathprintnumber{\slopeP}$}

        \pgfplotstablecreatecol[create col/expr={abs(\thisrow{ePw2})}]{absePw2}\ConvWeak
        \addplot[name path,only marks,brown,mark=diamond*] table[x ={h}, y={absePw2}]{\ConvWeak};
        \addplot[mark=none, brown,forget plot,dashed] table[x={h}, y={create col/linear regression={y=absePw2}}]{\ConvWeak};
        \xdef\slopePw{\pgfplotstableregressiona}
        \addlegendentry{Platen 2, $\tilde{p}=\pgfmathprintnumber{\slopePw}$}

    \end{axis}
\end{tikzpicture}
    \caption{{Weak convergence result for $\lambda=1$: error of approximating $\E X_1^2$ vs.\ step size}}\label{fig:Weakconvergence}
\end{subfigure}
\vspace{10pt}
\begin{subfigure}{\textwidth}
\begin{tikzpicture}
    \begin{axis}[
        name=plot2,
        anchor=above north east,
        width=0.9\textwidth,
        height=0.2\textheight,
        axis x line=bottom,
        xmode = log,
        log basis x=2,
        axis y line=left,
        ymode = log,
        log basis y=2,
        xlabel=$h$,
        ylabel={$| \hat{\E}(Y^2_{1,N}) - \hat{\E}(X_1(T_N)^2)|$},
        domain=0:1,
        legend style={at={(0.5,1.3)},anchor=north},
        legend columns=3]
        \pgfplotstablecreatecol[create col/expr={abs(\thisrow{eEDSL})}]{abseEDSL}\ConvWeakFive
        \addplot[name path,only marks,blue, mark=*] table[x ={h}, y={abseEDSL}]{\ConvWeakFive};
        \addplot[mark=none, blue,forget plot] table[x={h}, y={create col/linear regression={y=abseEDSL}}]{\ConvWeakFive};
        \xdef\slopeEDSL{\pgfplotstableregressiona}
        \addlegendentry{EM DSL, $\tilde{p}=\pgfmathprintnumber{\slopeEDSL}$}

        \pgfplotstablecreatecol[create col/expr={abs(\thisrow{ePDSL})}]{absePDSL}\ConvWeakFive
        \addplot[name path,only marks,red, mark=*] table[x ={h}, y={absePDSL}]{\ConvWeakFive};
        \addplot[mark=none, red,forget plot] table[x={h}, y={create col/linear regression={y=absePDSL}}]{\ConvWeakFive};
        \xdef\slopePDSL{\pgfplotstableregressiona}
        \addlegendentry{Platen DSL, $\tilde{p}=\pgfmathprintnumber{\slopePDSL}$}

        \pgfplotstablecreatecol[create col/expr={abs(\thisrow{ePw2DSL})}]{absePw2DSL}\ConvWeakFive
        \addplot[name path,only marks,brown, mark=*] table[x ={h}, y={absePw2DSL}]{\ConvWeakFive};
        \addplot[mark=none, brown,forget plot] table[x={h}, y={create col/linear regression={y=absePw2DSL}}]{\ConvWeakFive};
        \xdef\slopePwDSL{\pgfplotstableregressiona}
        \addlegendentry{Platen 2 DSL, $\tilde{p}=\pgfmathprintnumber{\slopePwDSL}$}

        \pgfplotstablecreatecol[create col/expr={abs(\thisrow{eE})}]{abseE}\ConvWeakFive
        \addplot[name path,only marks,blue,mark=diamond*] table[x ={h}, y={abseE}]{\ConvWeakFive};
        \addplot[mark=none, blue,forget plot,dashed] table[x={h}, y={create col/linear regression={y=abseE}}]{\ConvWeakFive};
        \xdef\slopeE{\pgfplotstableregressiona}
        \addlegendentry{EM, $\tilde{p}=\pgfmathprintnumber{\slopeE}$}

        \pgfplotstablecreatecol[create col/expr={abs(\thisrow{eP})}]{abseP}\ConvWeakFive
        \addplot[name path,only marks,red,mark=diamond*] table[x ={h}, y={abseP}]{\ConvWeakFive};
        \addplot[mark=none, red,forget plot, dashed] table[x={h}, y={create col/linear regression={y=abseP}}]{\ConvWeakFive};
        \xdef\slopeP{\pgfplotstableregressiona}
        \addlegendentry{Platen, $\tilde{p}=\pgfmathprintnumber{\slopeP}$}

        \pgfplotstablecreatecol[create col/expr={abs(\thisrow{ePw2})}]{absePw2}\ConvWeakFive
        \addplot[name path,only marks,brown,mark=diamond*] table[x ={h}, y={absePw2}]{\ConvWeakFive};
        \addplot[mark=none, brown,forget plot,dashed] table[x={h}, y={create col/linear regression={y=absePw2}}]{\ConvWeakFive};
        \xdef\slopePw{\pgfplotstableregressiona}
        \addlegendentry{Platen 2, $\tilde{p}=\pgfmathprintnumber{\slopePw}$}

    \end{axis}
\end{tikzpicture}
    \caption{{Weak convergence result for $\lambda=5$: error of approximating $\E X_1^2$  vs.\ step size}}\label{fig:Weakconvergence5}
\end{subfigure}
\vspace{10pt}
\begin{subfigure}{\textwidth}
\begin{center}
\begin{tikzpicture}
    \newenvironment{customlegend}[1][]{%
        \begingroup
        \csname pgfplots@init@cleared@structures\endcsname
        \pgfplotsset{#1}%
    }{%
        % draws the legend:
        \csname pgfplots@createlegend\endcsname
        \endgroup
    }%
\def\addlegendimage{\csname pgfplots@addlegendimage\endcsname}

\begin{customlegend}[legend columns=3,legend style={align=left,draw=none,column sep=2ex},legend entries={EM DSL, Platen DSL, Platen $2$ DSL,EM, Platen, Platen $2$}]
        \addlegendimage{blue,mark=*}
        \addlegendimage{red,mark=*}
        \addlegendimage{brown,mark=*}
        \addlegendimage{blue,dashed,mark=diamond*}
        \addlegendimage{red,dashed,mark=diamond*}
        \addlegendimage{brown,dashed,mark=diamond*}
\end{customlegend}
\end{tikzpicture}
\end{center}
\end{subfigure}
\begin{subfigure}{0.45\textwidth}
\xdef\basis{2}
\begin{tikzpicture}
    \begin{axis}[
        name=plot1,
        width=0.9\textwidth,
        height=0.29\textheight,
        axis x line=bottom,
        xmode = log,
        log basis x=\basis,
        axis y line=left,
        ymode = log,
        log basis y=\basis,
        title={},
        ylabel={Wall-clock time $[s]$},
        xlabel={$|\hat{\E}(Y^2_{1,N}) - \hat{\E}(X_1(T_N)^2)|$},
        legend pos=north east]

        \pgfplotstablecreatecol[create col/expr={abs(\thisrow{eEDSL})}]{abseEDSL}\Timings
        \addplot[name path,blue,mark=*] table[x ={abseEDSL}, y={tEDSL}]{\Timings};

        \pgfplotstablecreatecol[create col/expr={abs(\thisrow{ePDSL})}]{absePDSL}\Timings
        \addplot[name path,red,mark=*] table[x ={absePDSL}, y={tPDSL}]{\Timings};

        \pgfplotstablecreatecol[create col/expr={abs(\thisrow{ePw2DSL})}]{absePw2DSL}\Timings
        \addplot[name path,brown,mark=*] table[x ={absePw2DSL}, y={tPw2DSL}]{\Timings};

        %%%%%%%%%%%%%%%%
        \pgfplotstablecreatecol[create col/expr={abs(\thisrow{eE})}]{abseE}\Timings
        \addplot[name path,dashed,blue,mark=diamond*] table[x ={abseE}, y={tE}]{\Timings};

        \pgfplotstablecreatecol[create col/expr={abs(\thisrow{eP})}]{abseP}\Timings
        \addplot[name path,dashed,red,mark=diamond*] table[x ={abseP}, y={tP}]{\Timings};

        \pgfplotstablecreatecol[create col/expr={abs(\thisrow{ePw2})}]{absePw2}\Timings
        \addplot[name path,dashed,brown,mark=diamond*] table[x ={absePw2}, y={tPw2}]{\Timings};
    \end{axis}
\end{tikzpicture}
    \caption{{Wall-clock time per batch of 15000 paths vs.\ accuracy for $\lambda=1$}}\label{fig:weakefficiency1}
\end{subfigure}\hspace*{\fill}
\begin{subfigure}{0.45\textwidth}
\xdef\basis{2}
\begin{tikzpicture}
    \begin{axis}[
        name=plot1,
        width=0.9\textwidth,
        height=0.29\textheight,
        axis x line=bottom,
        xmode = log,
        log basis x=\basis,
        xmax = 0.001,
        axis y line=left,
        ymode = log,
        log basis y=\basis,
        title={},
        ylabel={Wall-clock time $[s]$},
        xlabel={$|\hat{\E}(Y^2_{1,N}) - \hat{\E} (X_1(T_N)^2)|$},
        legend pos=north east]

        \pgfplotstablecreatecol[create col/expr={abs(\thisrow{eEDSL})}]{abseEDSL}\TimingsFive
        \addplot[name path,blue,mark=*] table[x ={abseEDSL}, y={tEDSL}]{\TimingsFive};

        \pgfplotstablecreatecol[create col/expr={abs(\thisrow{ePDSL})}]{absePDSL}\TimingsFive
        \addplot[name path,red,mark=*] table[x ={absePDSL}, y={tPDSL}]{\TimingsFive};

        \pgfplotstablecreatecol[create col/expr={abs(\thisrow{ePw2DSL})}]{absePw2DSL}\TimingsFive
        \addplot[name path,brown,mark=*] table[x ={absePw2DSL}, y={tPw2DSL}]{\TimingsFive};

        %%%%%%%%%%%%%%%%
        \pgfplotstablecreatecol[create col/expr={abs(\thisrow{eE})}]{abseE}\TimingsFive
        \addplot[name path,dashed,blue,mark=diamond*] table[x ={abseE}, y={tE}]{\TimingsFive};

        \pgfplotstablecreatecol[create col/expr={abs(\thisrow{eP})}]{abseP}\TimingsFive
        \addplot[name path,dashed,red,mark=diamond*] table[x ={abseP}, y={tP}]{\TimingsFive};

        \pgfplotstablecreatecol[create col/expr={abs(\thisrow{ePw2})}]{absePw2}\TimingsFive
        \addplot[name path,dashed,brown,mark=diamond*] table[x ={absePw2}, y={tPw2}]{\TimingsFive};

    \end{axis}
\end{tikzpicture}
    \caption{{Wall-clock time per batch of 15000 paths vs.\ accuracy for $\lambda=5$}}\label{fig:weakefficiency5}
\end{subfigure}\hspace*{\fill}
\caption{{Weak convergence results for SDE \cref{eq:testSDEstrongconvergence}}}
\end{figure}
{We see that in this example, the {EM} DSL {scheme }and the Platen DSL scheme deliver very
  similar results. The same holds for the underlying schemes. For both the EM, Platen and Platen 2
  schemes, the error of the DSL versions is much smaller than the one of the underlying scheme. All
  schemes show clearly the expected order of convergence, except that for $\lambda=5$ and the
  considered step sizes, the Platen DSL scheme shows a somewhat smaller numerical order, and the
  Platen 2 DSL scheme exhibits one order higher than expected, caused by the order three dominating the order two error term.

In \cref{fig:weakefficiency1,fig:weakefficiency5} we depict the wall-clock time per batch of 15000 paths vs.\ the weak error {averaged over all batches}. Also here, the results for the EM DSL {scheme } and {the }Platen DSL {scheme} are very similar, and the same {holds} for the underlying schemes. The DSL schemes are in all cases more efficient than their underlying methods.
}

For an application to highly oscillatory differential equations formulated as Stratonovich SDEs, we refer to \cite{debrabantXXlsf}. We will in the following discuss linear stability for the Euler--Maruyama DSL and Platen DSL schemes.

\section{Linear mean-square stability analysis}\label{sec:MSstab}
In this section, we calculate the mean-square stability regions of the Euler--Maruyama DSL and Platen DSL schemes based on the standard linear test equation \cite{higham00msa}. Moreover we consider a higher-dimensional test-equation suggested by Buckwar and Kelly \cite{buckwar10tas,buckwar11acl,buckwar12nds}. We assume that for all the methods considered, $\Delta W^n_m$ is chosen as exact Wiener increment, $\Delta W^n_m=W_m(t_{n+1})-W_m(t_{n})$.

\subsection{Linear stability for multiplicative noise}
To analyze linear stability in the case of multiplicative noise we consider the test equation
\begin{equation}\label{eq:testlinstabmultnoise}
X(t)=X_0+\int_{0}^t(\lambda+\sigma) X(s)\ds+\int_{t_0}^t\mu X(s)\dW(s),
\end{equation}
with $X_0$ independent from $W$, $\lambda,\mu,\sigma\in\C$ and exact solution
\begin{equation}
X(t)=X_0e^{(\lambda+\sigma-\frac{\mu^2}2)t+\mu W(t)},
\end{equation}
which implies
\begin{align*}
\E(|X(t)|^2)=e^{(2\Re(\lambda+\sigma)+|\mu|^2)t}\E(|X_0|^2).
\end{align*}
Thus, the solution of test equation \eqref{eq:testlinstabmultnoise} is mean square stable,
\begin{equation*}
\lim_{t\to\infty}\E(|X(t)|^2)=0,
\end{equation*}
if and only if
\begin{equation}\label{eq:Equationstable}
2\Re(\lambda+\sigma)+|\mu|^2<0,
\end{equation}
see, e.\,g., \cite{saito96sao,higham00msa}.

\subsubsection{Analysis for Euler--Maruyama DSL method}
Application of the  Euler--Maruyama DSL method {with $A_0 = \lambda$ (and $A_1 = 0$, $g_0(t,x) = \sigma x$), $g_1(t,x) = \mu x$)} to \eqref{eq:testlinstabmultnoise} yields
\begin{equation*}
  Y_{n+1}=e^{\lambda h}(1+h\sigma+\mu\Delta W^n)Y_n,
\end{equation*}
and therefore
\begin{align*}
\E\left(|Y_{n+1}|^2\right)&=e^{2h\Re\lambda}\left(|1+h\sigma|^2+h|\mu|^2\right)\E\left(|Y_{n}|^2\right)\\
&=R(h\Re\lambda,h\sigma,\sqrt{h}|\mu|)\E\left(|Y_{n}|^2\right)
\end{align*}
with stability function
\[
R:~\R\times\C\times\R_+\to\R_+:~R(z,u,v)=e^{2z}(|1+u|^2+v^2),
\]
and a four{-}dimensional mean-square stability domain
\begin{equation}\label{eq:StabDomain}
S=\{(z,u,v)\in\R\times\C\times\R_+:~R(z,u,v)\leq1\},
\end{equation}
which for an A{-}stable method would be a superset of the mean-square stability domain of the exact solution,
\[
S^\star=\{(z,u,v)\in\R\times\C\times\R_+:~2z+2\Re u+v^2<0\}.
\]
For $(z,u,v)\in S^\star$ and $|u| \leq -{\sqrt{2}}z$ (implying $ z\leq 0$) it follows that
\[
|1+u|^2+v^2<1+|u|^2-2z\leq e^{-2z},
\]
and thus $(z,u,v)\in S$. So, we can conclude that if the SDE has a mean-square stable solution (i.\,e.\ fulfills \eqref{eq:Equationstable}) and fulfills
\begin{equation}\label{eq:suffcondexpeulermaruyamastable}
|\sigma|\leq -{\sqrt{2}}\Re(\lambda),
\end{equation}
then the Euler--Maruyama DSL scheme is mean-square stable independent of $h$.

Note also that the domain of mean-square stability of the scheme is, as the one of the exact solution, not dependent on $\Im\lambda$, in contrast to the situation for the conventional Euler--Maruyama method, where the stability condition would read
\[
|1+h(\sigma+\lambda)|^2+h|\mu|^2<1.
\]

\subsubsection{Analysis for Platen DSL scheme}
Application of the Platen DSL scheme to \eqref{eq:testlinstabmultnoise} yields
\begin{align*}
  H_2&=(1+h\sigma+\sqrt{h}\mu)Y_n,\\
  Y_{n+1}&=e^{\lambda h}\left(Y_n+h\sigma Y_n+\mu\Delta W^nY_n+\frac{\mu(H_2-Y_n)}{\sqrt{h}}\frac{(\Delta W^n)^2-h}2\right)
\end{align*}
and therefore
\begin{align*}
\E\left(|Y_{n+1}|^2\right)&=e^{2h\Re\lambda}\left(|1+h\sigma|^2+h|\mu|^2+\frac{h|\mu|^2}2|\sigma h+\mu\sqrt{h}|^2\right)\E\left(|Y_{n}|^2\right)\\
&=R(h\Re\lambda,h\sigma,\sqrt{h}\mu)\E\left(|Y_{n}|^2\right),
\end{align*}
with stability function
\[
R:~\R\times\C\times\C\to\R_+:~R(z,u,v)=e^{2z}\big(|1+u|^2+|v|^2(1+\frac12|u+v|^2)\big)
\]
and now a five{-}dimensional mean-square stability domain
\begin{equation}%\label{eq:StabDomain}
S=\{(z,u,v)\in\R\times\C\times\C:~R(z,u,v)\leq1\}.
\end{equation}
So, in contrast to the Euler--Maruyama DSL, now the mean-square stability of the method depends also on the argument of $\mu$, though still being independent of $\Im\lambda$ (in contrast to the conventional Platen scheme).

For $(z,u,v)\in \tilde{S}^\star$, where $\tilde{S}^\star$ is the canonical embedding of ${S}^\star$ into $\R\times\C^2$,
$|u|^2+\frac{|v|^4}2\leq2z^2$, $\Re(u\bar{v})\leq 0$ and $|v|^2|u|^2\leq-\frac83z^3$ it follows that
\begin{align*}
|1+u|^2+|v|^2(1+\frac12|u+v|^{{2}})&<1+|u|^2+\frac12|v|^2\underbrace{|u+v|^2}_{(|u|^2+|v|^2+2\Re(u\bar{v}))}-2z\\
&{=1-2z+\underbrace{|u|^2+\frac12|v|^4}_{\leq 2z^2}+\underbrace{\frac12|v|^2|u|^2}_{\leq{-}\frac43z^3}+|v|^2\underbrace{\Re(u\bar{v})}_{\leq0}}\\
&\leq e^{-2z},
\end{align*}
and thus $(z,u,v)\in S$. So, we can conclude that if the SDE has a mean-square stable solution (i.\,e.\ fulfills \eqref{eq:Equationstable}) and fulfills
\begin{equation}\label{eq:suffcondexPlatenstable}
|\sigma|^2+\frac{|\mu|^4}2\leq2\left(\Re(\lambda)\right)^2,~\Re(\sigma\bar{\mu})\leq 0,\text{ and }|\mu|^2|\sigma|^2\leq-\frac83\left(\Re(\lambda)\right)^3,
\end{equation}
then the Platen DSL scheme is mean-square stable independent of $h$.

Overall we have however to note that the conditions \eqref{eq:suffcondexpeulermaruyamastable} and \eqref{eq:suffcondexPlatenstable} are by no means necessary but only sufficient.

\subsection{Linear system stability for multiplicative noise}
{To study the linear system stability, we will follow the ideas outlined by
  Buckwar, Kelly and Sickenberger \cite{buckwar12asa,buckwar12nds,buckwar10tas}. In summary:  Given the linear
SDE}
\begin{equation}\label{eq:stabSystems}
\dX(t)=A_0X(t)\dt + \sum_{m=1}^M B_m X(t)\dW_m
\end{equation}
with $A_0,B_1,\dots,B_M\in\R^{d,d}$.
{Let $P(t)=\E \textrm{vec}(X(t)X(t)^{\top})$
    be the expectation of the  vectorization
  \footnote{{The vectorization of a $d\times d$ matrix $A=\{a_{i,j}\}_{i,j=1}^d$ is the $d^2$-dimensional
  vector given by $\textrm{vec}(A)=(a_{1,1},a_{2,1},\dots,a_{d,d-1},a_{d,d})^\top$.}}
    of the matrix process
  $X(t)X(t)^{\top}$. Then $P(t)$ is given by the solution of the $d^2$-dimensional linear ODE
  \[ dP(t) = SP(t)dt \]
where the}  mean-square stability matrix $S$ for this system is given by \cite{buckwar12asa}
\begin{equation}\label{eq:stabMatrix}
S = \Id \otimes A_0 + A_0 \otimes \Id + \sum_{m=1}^M B_m\otimes B_m,
\end{equation}
where $\otimes$ is the Kronecker product and $\Id$ denotes the $d$-dimensional unit matrix.
{
The zero solution $X(t)=0$ is  asymptotically mean-square stable if and only if all the eigenvalues of
$S$ have a negative real part, \cite[Lemma 3.3]{buckwar12asa}.}
%\begin{lemma}\label{lem:stabSDE}
%The equilibrium solution $X\equiv 0$ of SDE \cref{eq:stabSystems} is globally mean-square asymptotically stable if and only if all the eigenvalues of $S$ have a negative real part.
%\end{lemma}
{
Similarly, let $Y_{n}$ be the numerical solution obtained by applying a one-step method to \cref{eq:stabSystems}, such that
\[ Y_{n+1} = \Bd_n Y_n,\]
and let
$Q_n= \E \textrm{vec}(Y_n Y_n^{\top})$ be the expectation of the vectorization of the matrix process
$Y_{n}Y_{n}^T$. Then
\[ Q_{n+1} = \Smet Q_n, \qquad \text{with}\qquad  \Smet = \E \Bd_n\otimes\Bd_n,  \]
and the method is asymptotically mean-square stable if and only $\rho(\Smet)<1$
\cite[Lemma 3.4]{buckwar12asa}.}

Applying the Euler--Maruyama DSL scheme to \cref{eq:stabSystems}, then a single step is given by
\begin{equation*}\label{eq:Eonestepstab1}
Y_{n+1} = e^{A_0 h} \Ad_n Y_n, \qquad \Ad_n = \Id + \sum_{m=1}^M B_m \Delta W^n_m.
\end{equation*}
Let $\bar{A}=A_0h$ and $\bar{B}_m = B_m \sqrt{h}$, then {it follows from the above that }the stability matrix is given by
\begin{equation}\label{equ:stabilityMdEuler}
\Smet = e^{\bar{A}} \otimes e^{\bar{A}} \left( (\Id \otimes \Id) + \sum_{m=1}^M(\bar{B}_m\otimes \bar{B}_m)\right).
\end{equation}
Similarly, applying the Platen DSL scheme to \cref{eq:stabSystems} with $M=1$, we can write the one-step method as
\begin{align*}
Y_{n+1} &= e^{A_0 h} \Ad_n Y_n, \mbox{ with}\\
\Ad_n &= \Id + B_1 \Delta W_1^n + \frac{1}{2\sqrt{h}}\left(e^{-A_0h} B_1 e^{A_0h} (\Id + B_1 \sqrt{h})-B_1\right) \left( (\Delta W^n_m)^2 - h\right).
\end{align*}
Using the same definitions for $\bar{A}$ and $\bar{B}$ as above the stability matrix {of the
scheme} becomes
\begin{multline*}
\Smet = e^{\bar{A}} \otimes e^{\bar{A}} \Bigg( (\Id \otimes \Id) + (\bar{B}\otimes \bar{B})\ifnotarxiv\\\fi
         + \left(e^{-\bar{A}} \bar{B} e^{\bar{A}}(\Id + \bar{B})- \bar{B}\right) \otimes \left(e^{-\bar{A}} \bar{B} e^{\bar{A}}(\Id + \bar{B})- \bar{B}\right)\Bigg),
\end{multline*}
or with $\bar{C} :=e^{-\bar{A}} \bar{B} e^{\bar{A}}(\Id + \bar{B})- \bar{B}$
\begin{equation}\label{equ:stabilityMd}
	\Smet = e^{\bar{A}} \otimes e^{\bar{A}} \left( (\Id \otimes \Id) + (\bar{B}\otimes \bar{B}) + (\bar{C}\otimes \bar{C})\right).
\end{equation}
The methods are asymptotically stable if $\rho(\Smet)<1$. Note that if $\bar{A}$ and $\bar{B}$ commute, then $\bar{C}$ reduces to $\bar{C}=\bar{B}^2$, which corresponds to the term expected for the classical Platen scheme, whose stability matrix is
\begin{equation}\label{equ:stabilityImplPlaten}
	\Smet =  (\Id + \bar{A}) \otimes (\Id + \bar{A}) + (\bar{B}\otimes \bar{B}) + (\bar{B}^2\otimes \bar{B}^2).
\end{equation}
So in contrast to the relation between the stability regions of the standard Euler--Maruyama and the Euler--Maruyama DSL schemes, the commutator $[\bar{A},\bar{B}]$ influences the relation between the standard Platen and the Platen DSL schemes.

We will now calculate and plot the stability regions for some examples and also show some corresponding SDE simulations. We will compare three schemes:
\begin{itemize}
\item the Euler--Maruyama DSL scheme (\enquote{EM DSL}) \cref{equ:driftEuler},
\item the Platen DSL scheme (\enquote{Platen DSL}) \cref{equ:driftPlaten},
\item the implicit Platen strong order 1.0 scheme (\enquote{Impl.\ Platen}) \cite[Chapter 12.3.1]{kloeden99nso}.
\end{itemize}
\subsubsection{Linear test equation with orthogonal noise} \label{sec:BuckwarKelly}
We consider the linear test equation with {noise orthogonal to the flow $\left(X_1(t),X_2(t)\right)^{\top}$} \cite[Equ. 9]{buckwar12nds}
\begin{equation} \label{equ:LinDriftOrtNoiseNCommute}
\dmath \begin{pmatrix} X_1(t) \\ X_2(t)\end{pmatrix} =\underbrace{\begin{pmatrix} \lambda & b \\0 &\lambda \end{pmatrix}}_{=A_0} \begin{pmatrix} X_1(t) \\ X_2(t)\end{pmatrix} \dt + \underbrace{\begin{pmatrix} 0 & \sigma \\-\sigma &0 \end{pmatrix}}_{=B_1} \begin{pmatrix} X_1(t) \\ X_2(t)\end{pmatrix} \dW{(t)}
\end{equation}
{where $\lambda,b,\sigma\in\R$.}
With $\bar{A}=hA_0$, $\bar{B}=\sqrt{h}B_1$ and $\bar{C}$ as defined above we can now calculate the stability matrices {$\Smet(\lambda h,b h, \sigma\sqrt{h})$ according to}
\cref{equ:stabilityMdEuler,equ:stabilityMd} for the two Lawson schemes and {according to} \cref{equ:stabilityImplPlaten} for the implicit Platen scheme. {The mean square asymptotic stability region of the schemes is in each case given by $\{(\lambda h,bh,\sigma\sqrt{h})\in\R^3:~\rho(\Smet(\lambda h,bh,\sigma\sqrt{h}))<1\}.$}
{In }\namecref{fig:stabReg}~\hyperlink{fig:stabRegHyp}{\labelcref*{fig:stabReg}}{, we depict for various choices of $bh$ four slices of these stability regions} {and the corresponding stability region of the exact solution}.

\pgfplotstableread[col sep=comma]{Data/BuckwarKelly/Region/LinearExactB0.csv}\comEx
\pgfplotstableread[col sep=comma]{Data/BuckwarKelly/Region/LinearEulerB0.csv}\comE
\pgfplotstableread[col sep=comma]{Data/BuckwarKelly/Region/LinearPlatenB0.csv}\comLa
\pgfplotstableread[col sep=comma]{Data/BuckwarKelly/Region/LinearImplB0.csv}\comIm

\pgfplotstableread[col sep=comma]{Data/BuckwarKelly/Region/LinearExactB1.csv}\NcomEx
\pgfplotstableread[col sep=comma]{Data/BuckwarKelly/Region/LinearEulerB1.csv}\NcomE
\pgfplotstableread[col sep=comma]{Data/BuckwarKelly/Region/LinearPlatenB1.csv}\NcomLa
\pgfplotstableread[col sep=comma]{Data/BuckwarKelly/Region/LinearImplB1.csv}\NcomIm

\pgfplotstableread[col sep=comma]{Data/BuckwarKelly/Region/LinearExactB2.csv}\NNcomEx
\pgfplotstableread[col sep=comma]{Data/BuckwarKelly/Region/LinearEulerB2.csv}\NNcomE
\pgfplotstableread[col sep=comma]{Data/BuckwarKelly/Region/LinearPlatenB2.csv}\NNcomLa
\pgfplotstableread[col sep=comma]{Data/BuckwarKelly/Region/LinearImplB2.csv}\NNcomIm

\pgfplotstableread[col sep=comma]{Data/BuckwarKelly/Region/LinearExactB5.csv}\NNNcomEx
\pgfplotstableread[col sep=comma]{Data/BuckwarKelly/Region/LinearEulerB5.csv}\NNNcomE
\pgfplotstableread[col sep=comma]{Data/BuckwarKelly/Region/LinearPlatenB5.csv}\NNNcomLa
\pgfplotstableread[col sep=comma]{Data/BuckwarKelly/Region/LinearImplB5.csv}\NNNcomIm

\hypertarget{fig:stabRegHyp}{}
\begin{figure}[ht!]
\begin{tikzpicture}
    \newenvironment{customlegend}[1][]{%
        \begingroup
        \csname pgfplots@init@cleared@structures\endcsname
        \pgfplotsset{#1}%
    }{%
        % draws the legend:
        \csname pgfplots@createlegend\endcsname
        \endgroup
    }%
\def\addlegendimage{\csname pgfplots@addlegendimage\endcsname}

\begin{customlegend}[legend columns=4,legend style={align=left,draw=none,column sep=2ex},legend entries={Impl. Platen, Platen DSL, EM DSL, Exact solution}]
        \addlegendimage{blue}
        \addlegendimage{red}
        \addlegendimage{green!50!black}
        \addlegendimage{black}
\end{customlegend}
\end{tikzpicture}

\begin{tikzpicture}
\begin{axis}[%
    name=plot1,
    width=0.45\textwidth,
    height=0.15\textheight,
    xmin = -3,
    ymax = 4,
    axis x line=bottom,
    axis y line=left,
    title={$b=0$},
    xlabel=$\lambda {h}$,
    ylabel={$\sigma^2 {h}$},
    legend pos=north west]

    \addplot+[name path=f1,mark=none,blue] table[x ={x1}, y expr={(\thisrow{x2})^2}]{\comIm};

    \addplot+[name path=f2,mark=none,red] table[x= {x1}, y expr={(\thisrow{x2})^2}]{\comLa};

    \addplot+[name path=f4,mark=none,green!50!black] table[x= {x1}, y expr={(\thisrow{x2})^2}]{\comE};

    \addplot+[name path=f3, mark=none, color=black] table[x={x1}, y expr={\thisrow{x2}^2}]{\comEx};

    \path[name path=axis] (axis cs:-3,0) -- (axis cs:0,0);
    \path[name path=axis2](0,0) -- (-3,0) -- (-3,5);% -- (axis cs:0,0);

    \addplot [
        thick,
        color=green,
        fill=green,
        fill opacity=0.10
    ]
    fill between[
        of=f4 and axis2,
    ];

       \addplot [
        thick,
        color=red,
        fill=red,
        fill opacity=0.15
    ]
    fill between[
        of=f2 and axis2,
        %soft clip={domain=-3:0},
    ];

       \addplot [
        thick,
        color=gray,
        fill=gray,
        fill opacity=0.15
    ]
    fill between[
        of=f3 and axis2,
        %soft clip={domain=-3:0},
    ];

    \addplot [
        thick,
        color=blue,
        fill=blue,
        fill opacity=0.2
    ]
    fill between[
        of=f1 and axis,
        soft clip={domain=-3:0},
    ];

\end{axis}

\begin{axis}[%
    name=plot2,
    at=(plot1.right of south east), anchor=left of south west,
    width=0.45\textwidth,
    height=0.15\textheight,
    xmin = -3,
    ymax = 4,
    axis x line=bottom,
    axis y line=left,
    title={$bh=1$},
    xlabel=$\lambda h$,
    ylabel={$\sigma^2 h$},
    legend pos=north west]

    \addplot+[name path=f1,mark=none,blue] table[x={x1}, y expr={\thisrow{x2}^2}]{\NcomIm};

    \addplot+[name path=f2,mark=none,red] table[x={x1}, y expr={\thisrow{x2}^2}]{\NcomLa};

    \addplot+[name path=f4,mark=none,green!50!black] table[x= {x1}, y expr={(\thisrow{x2})^2}]{\NcomE};

    \addplot+[name path=f3, mark=none, color=black] table[x={x1}, y expr={\thisrow{x2}^2}]{\NcomEx};

    \addplot[mark=*] coordinates {(-2.0,2.5)};
    \addplot[mark=*] coordinates {(-1,2.5)};

    \path[name path=axis] (axis cs:-3,0) -- (axis cs:0,0);
    \path[name path=axis2](0,0) -- (-3,0) -- (-3,5);% -- (axis cs:0,0);

    \addplot [
        thick,
        color=green,
        fill=green,
        fill opacity=0.10
    ]
    fill between[
        of=f4 and axis2,
    ];

       \addplot [
        thick,
        color=red,
        fill=red,
        fill opacity=0.15
    ]
    fill between[
        of=f2 and axis2,
        %soft clip={domain=-3:0},
    ];

       \addplot [
        thick,
        color=gray,
        fill=gray,
        fill opacity=0.15
    ]
    fill between[
        of=f3 and axis2,
    ];

    \addplot [
        thick,
        color=blue,
        fill=blue,
        fill opacity=0.2
    ]
    fill between[
        of=f1 and axis,
        soft clip={domain=-3:0},
    ];
\end{axis}

\begin{axis}[%
    name=plot3,
    at=(plot1.below south east), anchor=above north east,
    width=0.45\textwidth,
    height=0.15\textheight,
    xmin = -3,
    ymax = 4,
    axis x line=bottom,
    axis y line=left,
    title={$bh=2$},
    xlabel=$\lambda h$,
    ylabel={$\sigma^2 h$},
    legend pos=north west]

    \addplot+[name path=f1,mark=none,blue] table[x ={x1}, y expr={(\thisrow{x2})^2}]{\NNcomIm};
    \addplot+[name path=f2,mark=none,red] table[x= {x1}, y expr={(\thisrow{x2})^2}]{\NNcomLa};

    \addplot+[name path=f4,mark=none,green!50!black] table[x= {x1}, y expr={(\thisrow{x2})^2}]{\NNcomE};
    
    \addplot+[name path=f3, mark=none, color=black] table[x={x1}, y expr={\thisrow{x2}^2}]{\NNcomEx};
    
    \path[name path=axis] (axis cs:-3,0) -- (axis cs:0,0);
    \path[name path=axis2](0,0) -- (-3,0) -- (-3,5);% -- (axis cs:0,0);

    \addplot [
        thick,
        color=green,
        fill=green,
        fill opacity=0.10
    ]
    fill between[
        of=f4 and axis2,
    ];

       \addplot [
        thick,
        color=red,
        fill=red,
        fill opacity=0.15
    ]
    fill between[
        of=f2 and axis2,
    ];

       \addplot [
        thick,
        color=gray,
        fill=gray,
        fill opacity=0.15
    ]
    fill between[
        of=f3 and axis2,
    ];

    \addplot [
        thick,
        color=blue,
        fill=blue,
        fill opacity=0.2
    ]
    fill between[
        of=f1 and axis,
        soft clip={domain=-3:0},
    ];

\end{axis}

\begin{axis}[%
    name=plot4,
    at=(plot3.right of south east), anchor=left of south west,
    width=0.45\textwidth,
    height=0.15\textheight,
    xmin = -3,
    ymax = 4,
    axis x line=bottom,
    axis y line=left,
    title={$bh=5$},
    xlabel=$\lambda h$,
    ylabel={$\sigma^2 h$},
    legend pos=north east]

    \addplot+[name path=f1,mark=none, blue] table[x={x1}, y expr={\thisrow{x2}^2}]{\NNNcomIm};
    
    \addplot+[name path=f2,mark=none, red] table[x={x1}, y expr={\thisrow{x2}^2}]{\NNNcomLa};
    
    \addplot+[name path=f4,mark=none,green!50!black] table[x= {x1}, y expr={(\thisrow{x2})^2}]{\NNNcomE};
    
    \addplot+[name path=f3, mark=none, color=black] table[x={x1}, y expr={\thisrow{x2}^2}]{\NNNcomEx};
    
    \path[name path=axis] (axis cs:-3,0) -- (axis cs:0,0);
    \path[name path=axis2](0,0) -- (-3,0) -- (-3,5);% -- (axis cs:0,0);

    \addplot [
        thick,
        color=green,
        fill=green,
        fill opacity=0.10
    ]
    fill between[
        of=f4 and axis2,
    ];

       \addplot [
        thick,
        color=red,
        fill=red,
        fill opacity=0.15
    ]
    fill between[
        of=f2 and axis2,
    ];

       \addplot [
        thick,
        color=gray,
        fill=gray,
        fill opacity=0.15
    ]
    fill between[
        of=f3 and axis2,
    ];

    \addplot [
        thick,
        color=blue,
        fill=blue,
        fill opacity=0.2
    ]
    fill between[
        of=f1 and axis,
        soft clip={domain=-3:0},
    ];

\end{axis}
\end{tikzpicture}
    \caption{Stability regions of the Euler--Maruyama DSL, Platen DSL and implicit Platen scheme applied to the linear test equation. The coloured areas show the direction where $\rho(\Smet)<1$. The two marks mark the values for which we do the numerical simulation.\label{fig:stabReg}}
\end{figure}
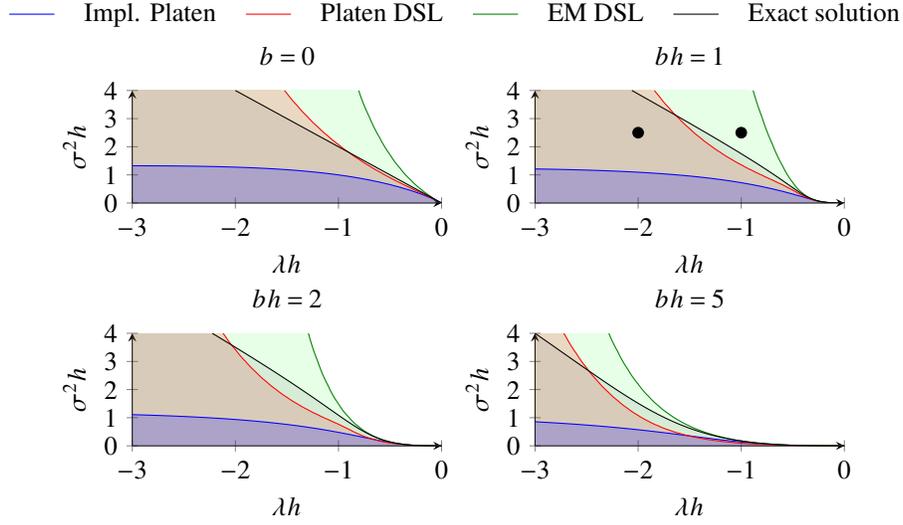
In \namecref{fig:stabReg}~\hyperlink{fig:stabRegHyp}{\labelcref*{fig:stabReg}} we see that the Euler--Maruyama DSL scheme {has the best stability properties} for the presented choices of $b{h}$, $\lambda{h}$ and ${\sigma^2}{h}$, but in some cases might over-stabilise, in particular for small values of $b{h}$. {Among} the three schemes, the Platen DSL scheme {reproduces} the stability region of the exact solution the best, especially for smaller values of $b{h}$. The implicit Platen {scheme} shows in all cases a lacking satisfactory stabilising effect in ${\sigma^2}{h}$, which is expected as this scheme is only implicit in the drift.

To demonstrate this effect{,} we now consider the situation that $b = 10$, $\lambda\in\{-20,-10\}$, $\sigma^2=25$ and $h=0.1$ corresponding to the two points in the top right figure of \namecref{fig:stabReg}~\hyperlink{fig:stabRegHyp}{\labelcref*{fig:stabReg}}. From this stability plot, we expect that for $\lambda h= -2.0$ and $\sigma^2h=2.5$ the two Lawson schemes perform better than the implicit Platen scheme, while for $\lambda h= -1.0$ we expect that the Euler--Maruyama DSL scheme over-stabilises and thus converges to $0$, whereas the two Platen schemes diverge. We use the initial value $X_0=(1,1)^T$ and simulate ${10^6}$ paths.

\pgfplotstableread[col sep=comma]{Data/BuckwarKelly/BKOutsideStabReg.csv}\outside
\pgfplotstableread[col sep=comma]{Data/BuckwarKelly/BKInsideStabReg.csv}\inside

\pgfplotstableread[col sep=comma]{Data/BuckwarKelly/ODE/BKOutsideStabRegRef.csv}\outsideref
\pgfplotstableread[col sep=comma]{Data/BuckwarKelly/ODE/BKInsideStabRegRef.csv}\insideref

\hypertarget{fig:numRes2Hyp}{}
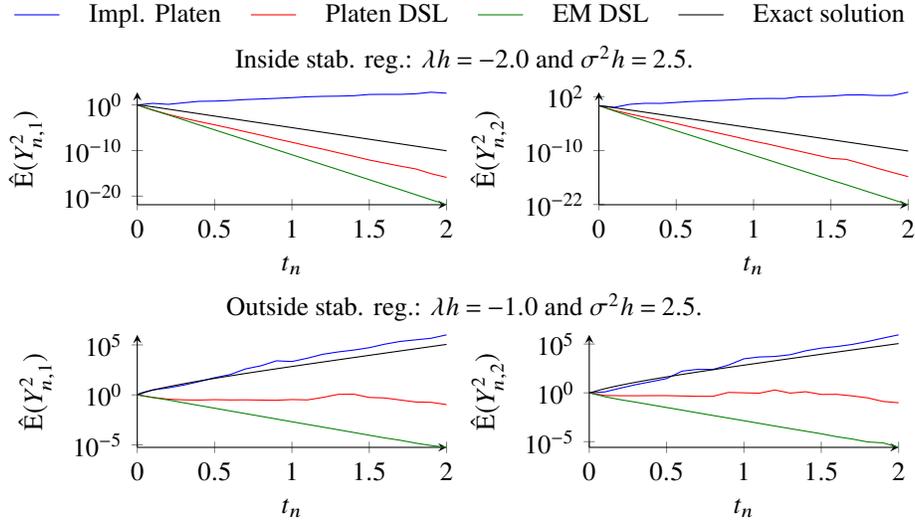
\begin{figure}[ht!]
\begin{center}
\begin{tikzpicture}
    \newenvironment{customlegend}[1][]{%
        \begingroup
        % inits/clears the lists (which might be populated from previous
        % axes):
        \csname pgfplots@init@cleared@structures\endcsname
        \pgfplotsset{#1}%
    }{%
        % draws the legend:
        \csname pgfplots@createlegend\endcsname
        \endgroup
    }%
\def\addlegendimage{\csname pgfplots@addlegendimage\endcsname}

\begin{customlegend}[legend columns=4,legend style={align=left,draw=none,column sep=2ex},legend entries={Impl. Platen, Platen DSL, EM DSL, Exact solution}]
        \addlegendimage{blue}
        \addlegendimage{red}
        \addlegendimage{green!50!black}
        \addlegendimage{black}
\end{customlegend}
\end{tikzpicture}

Inside stab. reg.: $\lambda h = -2.0$ and $\sigma^2 h = 2.5$.

\begin{tikzpicture}
\begin{axis}[%
    name=plot1,
    width=0.45\textwidth,
    height=0.15\textheight,
    axis x line=bottom,
    axis y line=left,
    ymode = log,
    title={ },
    xlabel=${t_n}$,
    ylabel={${\hat{\E}(Y_{n,1}^2)}$},
    legend pos=north east]

    \addplot+[mark=none] table[x={t}, y={impl1}]{\inside};
    
    \addplot+[mark=none] table[x={t}, y={Platen1}]{\inside};
    
    \addplot+[mark=none,green!50!black] table[x={t}, y={euler1}]{\inside};
    
    \addplot+[mark=none] table[x={t}, y={ref1}]{\insideref};

\end{axis}

\begin{axis}[%
    name=plot2,
    at=(plot1.right of south east), anchor=left of south west,
    width=0.45\textwidth,
    height=0.15\textheight,
    axis x line=bottom,
    axis y line=left,
    ymode = log,
    title={ },
    xlabel=${t_n}$,
    ylabel={${\hat{\E}(Y_{n,2}^2)}$},
    legend pos=north east]

    \addplot+[mark=none] table[x={t}, y={impl2}]{\inside};
    
    \addplot+[mark=none] table[x={t}, y={Platen2}]{\inside};
    
    \addplot+[mark=none,green!50!black] table[x={t}, y={euler2}]{\inside};
    
    \addplot+[mark=none] table[x={t}, y={ref2}]{\insideref};
    
\end{axis}
\end{tikzpicture}

Outside stab. reg.: $\lambda h=-1.0$ and $\sigma^2 h=2.5$.

\begin{tikzpicture}
\begin{axis}[%
    name=plot1,
    width=0.45\textwidth,
    height=0.15\textheight,
    axis x line=bottom,
    axis y line=left,
    ymode = log,
    title={ },
    xlabel=${t_n}$,
    ylabel={${\hat{\E}(Y_{n,1}^2)}$},
    legend pos=north west]

    \addplot+[mark=none] table[x={t}, y={impl1}]{\outside};

    \addplot+[mark=none] table[x={t}, y={Platen1}]{\outside};

    \addplot+[mark=none,green!50!black] table[x={t}, y={euler1}]{\outside};

    \addplot+[mark=none] table[x={t}, y={ref1}]{\outsideref};

\end{axis}

\begin{axis}[%
    name=plot2,
    at=(plot1.right of south east), anchor=left of south west,
    width=0.45\textwidth,
    height=0.15\textheight,
    axis x line=bottom,
    axis y line=left,
    ymode = log,
    title={ },
    xlabel=${t_n}$,
    ylabel={${\hat{\E}(Y_{n,2}^2)}$},
    legend pos=north west]

    \addplot+[mark=none] table[x={t}, y={impl2}]{\outside};
    
    \addplot+[mark=none] table[x={t}, y={Platen2}]{\outside};
    
    \addplot+[mark=none,green!50!black] table[x={t}, y={euler2}]{\outside};
    
    \addplot+[mark=none] table[x={t}, y={ref2}]{\outsideref};
    \end{axis}
\end{tikzpicture}
\end{center}
\caption{Numerical results for \cref{equ:LinDriftOrtNoiseNCommute} on the boundary of the stability region of the exact solution}\label{fig:numRes2}
\end{figure}
In \namecref{fig:numRes2}~\hyperlink{fig:numRes2Hyp}{\labelcref*{fig:numRes2}} we plot the evolution of the second moment of $X_1$ and $X_2$. To calculate the exact moments, we derived the ODE system for ${\E}(X_1^2)$, ${\E}(X_2^2)$ and ${\E}(X_1 X_2)$ using \ito's formula and solved it in Matlab using "ode15s" \cite{shampine97tmo} with an absolute tolerance of ${10^{-14}}$. We see that both the Platen DSL and Euler--Maruyama DSL schemes are indeed stable just inside their stability region, $\lambda h=-2.0$ and $\sigma^2h=2.5$, whereas the implicit Platen scheme fails to be stable as predicted by its stability region. Conversely, we see that when $\lambda h=-1.0$  and $\sigma^2h=2.5$, then the implicit Platen scheme diverges due to the parameters lying outside of its stability region, while the Platen DSL scheme remains close to ${1}$, indicating that {$(\lambda h,\sigma^2h)$ is} near the boundary of {the scheme's} stability region with $\rho(\Smet)$ close to $1$. The Euler--Maruyama DSL scheme converges to $0$, as the chosen parameters are still inside of its stability region.
\subsubsection{Damped and driven {oscillators}}
In the above experiments, we saw that the implicit Platen scheme, in contrast to the Lawson schemes, might fail to stabilise the SDE when the stiffness comes from the diffusion. However, for the damped/driven oscillator, we will see that eigenvalues with large complex parts might make the implicit Platen over-stabilise.

We consider the oscillator
\begin{equation} \label{equ:dampedOsc}
\dmath \begin{pmatrix} X_1(t) \\ X_2(t)\end{pmatrix} =\underbrace{\begin{pmatrix} \lambda & \omega^2 \\-\omega^2 &\lambda \end{pmatrix}}_{=A_0} \begin{pmatrix} X_1(t) \\ X_2(t)\end{pmatrix} \dt + \underbrace{\begin{pmatrix} 0 & \sigma \\-\sigma &0 \end{pmatrix}}_{=B_1} \begin{pmatrix} X_1(t) \\ X_2(t)\end{pmatrix} \dW{(t)}
\end{equation}
{with $\lambda,\omega,\sigma\in\R$. }For this system{,} the matrices needed to compute the stability domains \cref{equ:stabilityMdEuler,equ:stabilityMd,equ:stabilityImplPlaten} are $\bar{A}=hA_0$, $\bar{B}=\sqrt{h}B_1$, and $\bar{C}=-h\sigma^2\Id$.
We note that with $\lambda$ and $\omega$ we can control the size of the real and imaginary part of the eigenvalues respectively of the matrix $A_0$. Now plotting the stability regions, we obtain \cref{fig:OscStab}.

\pgfplotstableread[col sep=comma]{Data/Oscillator/Region/RotationExactOmegaPi.csv}\oscEx
\pgfplotstableread[col sep=comma]{Data/Oscillator/Region/RotationLawsonOmegaPi.csv}\oscLa
\pgfplotstableread[col sep=comma]{Data/Oscillator/Region/RotationImplOmegaPi.csv}\oscIm
\pgfplotstableread[col sep=comma]{Data/Oscillator/Region/RotationEulerOmegaPi.csv}\oscE

\pgfplotstableread[col sep=comma]{Data/Oscillator/Region/RotationExactOmega10Pi.csv}\hoscEx
\pgfplotstableread[col sep=comma]{Data/Oscillator/Region/RotationLawsonOmega10Pi.csv}\hoscLa
\pgfplotstableread[col sep=comma]{Data/Oscillator/Region/RotationImplOmega10Pi.csv}\hoscIm
\pgfplotstableread[col sep=comma]{Data/Oscillator/Region/RotationEulerOmega10Pi.csv}\hoscE

\begin{figure}[ht!]
\begin{tikzpicture}
    \newenvironment{customlegend}[1][]{%
        \begingroup
        % inits/clears the lists (which might be populated from previous
        % axes):
        \csname pgfplots@init@cleared@structures\endcsname
        \pgfplotsset{#1}%
    }{%
        % draws the legend:
        \csname pgfplots@createlegend\endcsname
        \endgroup
    }%
\def\addlegendimage{\csname pgfplots@addlegendimage\endcsname}

\begin{customlegend}[legend columns=4,legend style={align=left,draw=none,column sep=2ex},legend entries={Impl. Platen, Platen DSL, EM DSL, Exact solution}]
        \addlegendimage{blue}
        \addlegendimage{red}
        \addlegendimage{green!50!black}
        \addlegendimage{black}
\end{customlegend}
\end{tikzpicture}

\begin{tikzpicture}
\begin{axis}[%
    name=plot1,
    width=0.45\textwidth,
    height=0.2\textheight,
    xmin = -1,
    xmax = 0,
    ymax = 3,
    ymin=0,
    xtick={0,-0.25,-0.5,-0.75,-1},
    axis x line=bottom,
    axis y line=left,
    title={$\omega^2h=\pi$},
    xlabel=$\lambda h$,
    ylabel={$\sigma^2 h$},
    legend pos=north east]

    \addplot+[name path=f1,mark=none,blue] table[x={x1}, y expr={\thisrow{x2}^2}]{\oscIm};
    
    \addplot+[name path=f2,mark=none,red] table[x={x1}, y expr={\thisrow{x2}^2}]{\oscLa};
    
    \addplot+[name path=f4,mark=none,green!50!black] table[x={x1}, y expr={\thisrow{x2}^2}]{\oscE};

    \addplot+[name path=f3, mark=none, color=black] table[x={x1}, y expr={\thisrow{x2}^2}]{\oscEx};
    
    \addplot[mark=*] coordinates {(-0.1,0.4)};
    \addplot[mark=*] coordinates {(-0.3,0.4)};

    \path[name path=axis] (axis cs:-3,0) -- (axis cs:0,0);
    \path[name path=axis2](0,0) -- (-1,0) -- (-1,4);% -- (axis cs:0,0);

    \addplot [
        thick,
        color=green,
        fill=green,
        fill opacity=0.10
    ]
    fill between[
        of=f4 and axis2,
    ];

       \addplot [
        thick,
        color=red,
        fill=red,
        fill opacity=0.15
    ]
    fill between[
        of=f2 and axis,
        soft clip={domain=-3:0},
    ];

       \addplot [
        thick,
        color=gray,
        fill=gray,
        fill opacity=0.15
    ]
    fill between[
        of=f3 and axis,
        soft clip={domain=0:-3},
    ];

    \addplot [
        thick,
        color=blue,
        fill=blue,
        fill opacity=0.2
    ]
    fill between[
        of=f1 and axis,
        soft clip={domain=-3:0},
    ];
\end{axis}

\begin{axis}[
    name=plot2,
    at=(plot1.right of south east), anchor=left of south west,
    width=0.45\textwidth,
    height=0.2\textheight,
    xmin = -1,
    xmax = 0,
    ymax = 3,
    ymin=0,
    axis x line=bottom,
    axis y line=left,
    title={$\omega^2h=10\pi $},
    xlabel=$\lambda h$,
    ylabel={$\sigma^2 h$},
    legend pos=south west]

    \addplot+[name path=f1,mark=none,blue] table[x={x1}, y expr={\thisrow{x2}^2}]{\hoscIm};
    %\addlegendentry{Implicit}

    \addplot+[name path=f2,mark=none,red] table[x={x1}, y expr={\thisrow{x2}^2}]{\hoscLa};
    
    \addplot+[name path=f4,mark=none,green!50!black] table[x={x1}, y expr={\thisrow{x2}^2}]{\hoscE};
    
    \addplot+[name path=f3, mark=none, color=black] table[x={x1}, y expr={\thisrow{x2}^2}]{\hoscEx};
    
    \path[name path=axis] (axis cs:-1,0) -- (axis cs:0,0);
	\path[name path=axis2](0,0) -- (-1,0) -- (-1,4);% -- (axis cs:0,0);

    \addplot [
        thick,
        color=green,
        fill=green,
        fill opacity=0.10
    ]
    fill between[
        of=f4 and axis2,
    
    ];

       \addplot [
        thick,
        color=red,
        fill=red,
        fill opacity=0.15
    ]
    fill between[
        of=f2 and axis,
        soft clip={domain=-1:0},
    ];

       \addplot [
        thick,
        color=gray,
        fill=gray,
        fill opacity=0.15
    ]
    fill between[
        of=f3 and axis,
        soft clip={domain=0:-1},
    ];

    \addplot [
        thick,
        color=blue,
        fill=blue,
        fill opacity=0.2
    ]
    fill between[
        of=f1 and axis,
        soft clip={domain=-1:0},
    ];

\end{axis}
\end{tikzpicture}
\caption{Stability regions of the Euler--Maruyama DSL, Platen DSL and implicit Platen scheme applied to the driven/damped oscillator. The coloured areas show the direction where $\rho(\Smet)<1$. The two marks mark for which values we simulate the problem.} \label{fig:OscStab}
\end{figure}
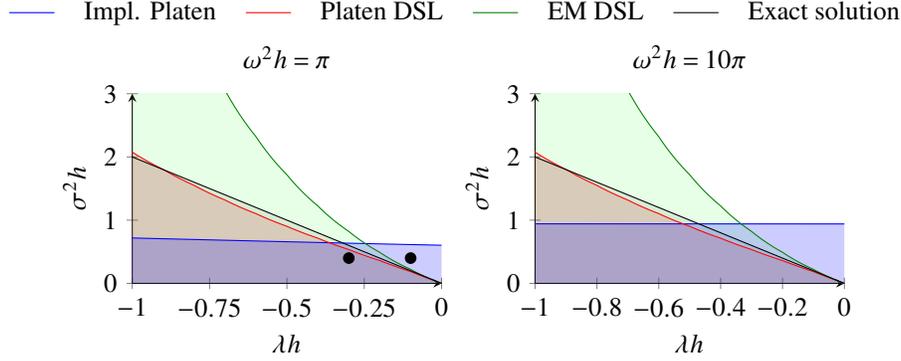

In \cref{fig:OscStab} we see that, as for the scalar test equation \eqref{eq:testlinstabmultnoise}, the stability domains of both Lawson schemes and the exact solution are independent of the imaginary part of the {eigenvalues of $A_0$}, {i.\,e.\ $\omega$}, whereas the one of the implicit Platen scheme depends on it.

To show the consequence, consider the oscillator with $X_0=(1,1)^\top$, $\omega^2=10\pi$, $\sigma^2=4$, $\lambda\in\{-1,-3\}$ and $h=0.1$ corresponding to the two marks in the left part of \cref{fig:OscStab}. For $\lambda h=0.1$ it is a driven oscillator, whereas for $\lambda h=0.3$ it is a damped one.

 We simulate ${\E}(|X_1|^2)$ using the Euler--Maruyama DSL, Platen DSL and implicit Platen schemes with $h=0.1$ and ${10000}$ paths. We also include the exact moment which, again, is obtained by solving the ODE system for ${\E}(X_1^2)$, ${\E}(X_2^2)$ and ${\E}(X_1 X_2)$ in Matlab using "ode15s" with an absolute tolerance of ${10^{-14}}$.

\pgfplotstableread[col sep=comma]{Data/Oscillator/DampedOscillator.csv}\osc
\pgfplotstableread[col sep=comma]{Data/Oscillator/DrivenOscillator.csv}\oscNS
\pgfplotstableread[col sep=comma]{Data/Oscillator/DampedOscillatorRef.csv}\oscref
\pgfplotstableread[col sep=comma]{Data/Oscillator/DrivenOscillatorRef.csv}\oscNSref

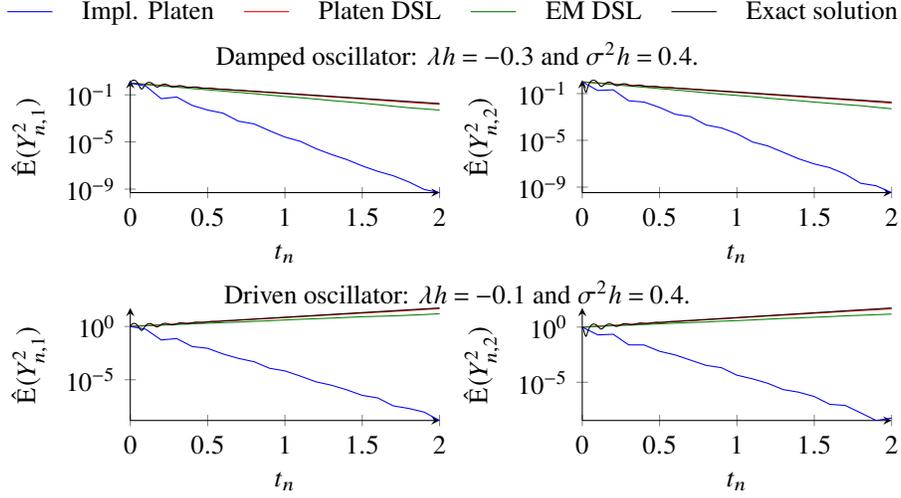
\begin{figure}[ht!]
\begin{center}
\begin{tikzpicture}
    \newenvironment{customlegend}[1][]{%
        \begingroup
    
        \csname pgfplots@init@cleared@structures\endcsname
        \pgfplotsset{#1}%
    }{%
        % draws the legend:
        \csname pgfplots@createlegend\endcsname
        \endgroup
    }%
\def\addlegendimage{\csname pgfplots@addlegendimage\endcsname}

\begin{customlegend}[legend columns=4,legend style={align=left,draw=none,column sep=2ex},legend entries={Impl. Platen, Platen DSL, EM DSL, Exact solution}]
        \addlegendimage{blue}
        \addlegendimage{red}
        \addlegendimage{green!50!black}
        \addlegendimage{black}
\end{customlegend}
\end{tikzpicture}

Damped oscillator: $\lambda h=-0.3$ and $\sigma^2 h = 0.4$.

\begin{tikzpicture}
\begin{axis}[%
    name=plot1,
    width=0.45\textwidth,
    height=0.15\textheight,
    axis x line=bottom,
    axis y line=left,
    xmax=2,
    ymode=log,
    title={ },
    xlabel={$t_n$},
    ylabel={${\hat{\E}(Y_{n,1}^2)}$},
    legend pos=north east]

    \addplot+[mark=none,red] table[x={t}, y={Platen1}]{\osc};
    
    \addplot+[mark=none,green!50!black] table[x={t}, y={euler1}]{\osc};
    
    \addplot+[mark=none,black] table[x={t}, y={ref1}]{\oscref};
    
    \addplot+[mark=none,blue] table[x={t}, y={impl1}]{\osc};
    
\end{axis}

\begin{axis}[%
    name=plot2,
    at=(plot1.right of south east), anchor=left of south west,
    width=0.45\textwidth,
    height=0.15\textheight,
    axis x line=bottom,
    axis y line=left,
    xmax=2,
    ymode=log,
    title={ },
    xlabel={$t_n$},
    ylabel={${\hat{\E}(Y_{n,2}^2)}$},
    legend pos=north east]

    \addplot+[mark=none,red] table[x={t}, y={Platen2}]{\osc};
    
    \addplot+[mark=none,green!50!black] table[x={t}, y={euler2}]{\osc};
    
    \addplot+[mark=none,blue] table[x={t}, y={impl2}]{\osc};
    
    \addplot+[mark=none,black] table[x={t}, y={ref2}]{\oscref};
    \end{axis}
\end{tikzpicture}

Driven {o}scillator: $\lambda h=-0.1$ and $\sigma^2 h = 0.4$.
\begin{tikzpicture}
\begin{axis}[%
    name=plot1,
    
    width=0.45\textwidth,
    height=0.15\textheight,
    axis x line=bottom,
    axis y line=left,
    xmax = 2,
    ymode=log,
    title={ },
    xlabel=${t_n}$,
    ylabel={${\hat{\E}(Y_{n,1}^2)}$},
    legend pos=south west]

    \addplot+[mark=none,red] table[x={t}, y={Platen1}]{\oscNS};

    \addplot+[mark=none,green!50!black] table[x={t}, y={euler1}]{\oscNS};
    
    \addplot+[mark=none,blue] table[x={t}, y={impl1}]{\oscNS};
    
    \addplot+[mark=none,black] table[x={t}, y={ref1}]{\oscNSref};
    \end{axis}

\begin{axis}[%
    name=plot2,
    at=(plot1.right of south east), anchor=left of south west,
    width=0.45\textwidth,
    height=0.15\textheight,
    axis x line=bottom,
    axis y line=left,
    xmax = 2,
    ymode = log,
    title={ },
    xlabel={$t_n$},
    ylabel={${\hat{\E}(Y_{n,2}^2)}$},
    legend pos=south west]

    \addplot+[mark=none,red] table[x={t}, y={Platen2}]{\oscNS};

    \addplot+[mark=none,green!50!black] table[x={t}, y={euler2}]{\oscNS};

    \addplot+[mark=none,blue] table[x={t}, y={impl2}]{\oscNS};

    \addplot+[mark=none,black] table[x={t}, y={ref2}]{\oscNSref};
    
\end{axis}

\end{tikzpicture}
\end{center}
\caption{Numerical results for the driven {and} damped {oscillators} using the Euler--Maruyama DSL, the Platen DSL, and the implicit Platen {schemes}.}\label{fig:dampedOsc}
\end{figure}

In \cref{fig:dampedOsc} we see that both for the damped {oscillator} and the driven oscillator, the implicit Platen scheme stabilises the numerical flow significantly more than the Platen DSL scheme. For the damped oscillator, the behaviour of all schemes is correct, and they all converge to $0$. The implicit Platen scheme does, however, converge way too fast compared to the exact solution. The Euler--Maruyama DSL scheme is also converging slightly too fast, but is closer to the exact solution. The Platen DSL scheme follows almost the exact solution.

For the driven oscillator, we see that the implicit Platen scheme shows the wrong behaviour, as if the oscillator would still be damped. The two Lawson schemes behave correctly, with the results of the Platen DSL scheme visually matching the exact solution and the approximations by the Euler--Maruyama DSL scheme being slightly off.

\section{Conclusion}
In this paper, we derived the general class of {SRK} Lawson schemes. We proved that, if the underlying SRK scheme is of mean-square order $p$, then the {SRK} Lawson scheme is of strong order $p$. Similarly, under the assumption that the linear diffusion terms included in the exponential operator are skew-symmetric, we proved that the {SRK} Lawson schemes also inherit the weak order of convergence from the underlying {SRK} scheme.

We performed a linear stability analysis for the Euler--Maruyama and Platen {DSL} schemes. In particular, we numerically demonstrated that the implicit Platen scheme might provide insufficient stabilisation when the destabilisation comes from the diffusion. However, the Euler--Maruyama and Platen {DSL} schemes provide adequate stabilisation, even though the exponential only includes the drift term. Conversely, we demonstrate that for driven oscillators with small diffusion terms, the implicit Platen {scheme} might over-stabilise and make it a damped oscillator, whereas the two {SL} schemes more accurately match the behaviour of the exact solution.

\section{Acknowledgement}
Nicky Cordua Mattsson would like to thank the SDU e-Science centre for partially funding his PhD and the Department of Mathematics at the Norwegian University of Science and Technology for kindly hosting him during his visit. The authors would like to thank two
anonymous reviewers for very detailed reading and the resulting many helpful comments.

\def\cprime{$'$} \providecommand{\de}[2]{#2}

{
\section{Appendix}
\subsection{Platen strong order 1.5 SL scheme} \label{sec:Platen15}
Writing the explicit order $1.5$ strong scheme by Platen for $M=1$ \cite[Eq. 11.2.1]{kloeden99nso} in the form \cref{equ:SRK}, we see that the coefficients $z_i^{m,n}$ of the scheme are given by
\begin{center}
\begin{tabular}{c|ll@{}}
		& $m=0$ 													& $m=1$ \\ \hline \hline
		$i=1$ & $z^{0,n}_1 = \frac{1}{2}h$ 									& $z^{1,n}_1 = \Delta W^n - \frac{1}{h}[\Delta W^n h-\Delta Z^n]$ \\[0.5em]
		$i=2$ & $z^{0,n}_2 =  \frac{1}{2\sqrt{h}}\Delta Z^n + \frac{1}{4}h$ & $z^{1,n}_2 = \frac{1}{h}\left[\frac{1}{2}\Delta W^n h-\Delta Z^n - \frac{1}{4} \{\frac{1}{3}(\Delta W^n)^2 - h\}\Delta W^n\right]$ \\[0.5em]
		$i=3$ & $z^{0,n}_3 =-\frac{1}{2\sqrt{h}}\Delta Z^n + \frac{1}{4}h$	& $z^{1,n}_3 = \frac{1}{h}\left[\frac{1}{2}\Delta W^n h-\Delta Z^n + \frac{1}{4} \{\frac{1}{3}(\Delta W^n)^2 - h\}\Delta W^n\right]$ \\[0.5em]
		$i=4$ & $z^{0,n}_4 = 0$ 												& $z^{1,n}_4 = \frac{1}{4h} \left[\frac{1}{3}(\Delta W^n)^2 - h\right]\Delta W^n$ \\[0.5em]
		$i=5$ & $z^{0,n}_5 = 0$ 												& $z^{1,n}_4 = -\frac{1}{4h} \left[\frac{1}{3}(\Delta W^n)^2 - h\right]\Delta W^n$
	\end{tabular}
\end{center}
{where
\[\Delta Z^n=\frac12h\Delta W^n + \frac{\sqrt{3}}6h^{3/2}U^n
\]
with $U ^n\sim\mathcal{N}(0,1)$.}
Similarly, the coefficients of $Z_{ij}^{m,n}$ are given by
\begin{center}
\begin{tabular}{c|lllll}
& $j=1$ & $j=2$ & $j=3$ & $j=4$ & $j=5$  \\ \hline \hline
	$i=1$&&&&&\\ \hline
	$m=0$ & $Z^{0,n}_{1,1} = 0$, & $Z^{0,n}_{1,2} = 0$, & $Z^{0,n}_{1,3} = 0$, & $Z^{0,n}_{1,4} = 0$, & $Z^{0,n}_{1,5} = 0$ \\[0.5em]
	$m=1$ &$Z^{1,n}_{1,1} = 0$, & $Z^{1,n}_{1,2} = 0$, & $Z^{1,n}_{1,3} = 0$, & $Z^{1,n}_{1,4} = 0$, & $Z^{1,n}_{1,5} = 0$ \\[0.5em]\hline
	$i=2$&&&&&\\ \hline
	$m=0$ &$Z^{0,n}_{2,1} = h$, & $Z^{0,n}_{2,2} = 0$, & $Z^{0,n}_{2,3} = 0$, & $Z^{0,n}_{2,4} = 0$, & $Z^{0,n}_{2,5} = 0$ \\[0.5em]
	$m=1$ &$Z^{1,n}_{2,1} = \sqrt{h}$, & $Z^{1,n}_{2,2} = 0$, & $Z^{1,n}_{2,3} = 0$, & $Z^{1,n}_{2,4} = 0$, & $Z^{1,n}_{2,5} = 0$ \\[0.5em]\hline
	$i=3$&&&&&\\ \hline
	$m=0$ &$Z^{0,n}_{1,1} = h$, & $Z^{0,n}_{1,2} = 0$, & $Z^{0,n}_{1,3} = 0$, & $Z^{0,n}_{1,4} = 0$, & $Z^{0,n}_{1,5} = 0$ \\[0.5em]
	$m=1$ &$Z^{1,n}_{1,1} = -\sqrt{h}$, & $Z^{1,n}_{1,2} = 0$, & $Z^{1,n}_{1,3} = 0$, & $Z^{1,n}_{1,4} = 0$, & $Z^{1,n}_{1,5} = 0$ \\[0.5em] \hline
	$i=4$&&&&&\\ \hline
	$m=0$ &$Z^{0,n}_{2,1} = h$, & $Z^{0,n}_{2,2} = 0$, & $Z^{0,n}_{2,3} = 0$, & $Z^{0,n}_{2,4} = 0$, & $Z^{0,n}_{2,5} = 0$ \\[0.5em]
	$m=1$ &$Z^{1,n}_{2,1} = \sqrt{h}$, & $Z^{1,n}_{2,2} = \sqrt{h}$, & $Z^{1,n}_{2,3} = 0$, & $Z^{1,n}_{2,4} = 0$, & $Z^{1,n}_{2,5} = 0$ \\[0.5em] \hline
	$i=5$&&&&&\\ \hline
	$m=0$ &$Z^{0,n}_{2,1} = h$, & $Z^{0,n}_{2,2} = 0$, & $Z^{0,n}_{2,3} = 0$, & $Z^{0,n}_{2,4} = 0$, & $Z^{0,n}_{2,5} = 0$ \\[0.5em]
	$m=1$ &$Z^{1,n}_{2,1} = \sqrt{h}$, & $Z^{1,n}_{2,2} = -\sqrt{h}$, & $Z^{1,n}_{2,3} = 0$, & $Z^{1,n}_{2,4} = 0$, & $Z^{1,n}_{2,5} = 0$
	\end{tabular}
\end{center}
Using the definitions of the coefficients $c_{m}^{n,i}$ and $\Delta L_{i}^n$ we calculate
\begin{center}
	\begin{tabular}{c|llcl}
		& $m=0$ 		& $m=1$ & & $\Delta L_i^n$ \\ \hline \hline
		$i=1$ & $c^{n,1}_0=0$ 	& $c^{n,1}_1=0$ 			& $\implies$ &$\Delta L_1^n = 0$ \\[0.5em]
		$i=2$ & $c^{n,2}_0=h$ 	& $c^{n,2}_1=\sqrt{h}$ 	& $\implies$ &$\Delta L_2^n = (A_0 -\gs A_1^2)h + A_1 \sqrt{h}$\\[0.5em]
		$i=3$ & $c^{n,3}_0=h$	& $c^{n,3}_1=-\sqrt{h}$ 	& $\implies$ &$\Delta L_3^n = (A_0 -\gs A_1^2)h - A_1 \sqrt{h}$\\[0.5em]
		$i=4$ & $c^{n,4}_0=h$ 	& $c^{n,4}_1=2\sqrt{h}$	& $\implies$ &$\Delta L_4^n = (A_0 -\gs A_1^2)h + 2A_1 \sqrt{h}$ \\[0.5em]
		$i=5$ & $c^{n,5}_0=h$ 	& $c^{n,5}_1=0$			& $\implies$ &$\Delta L_5^n = (A_0 -\gs A_1^2)h$
	\end{tabular}
\end{center}
Finally, using the definitions of $c_{m}^{n}$ and $\Delta L^n$ we calculate
$c^{n}_0=h$, $c^{n}_1=\Delta W^n$ and thus $\Delta L^n = (A_0 -\gs A_1^2)h + A_1\Delta W^n$.

With all the parameters of the scheme \cref{equ:SRKL} in place, the resulting Platen strong order 1.5 SL scheme is given by
\begin{align*}
	H_1 &= \bar{V}^n_n, \\
	H_2 &= \bar{V}^n_n + \gt^n_0 (t_n,H_1) h + \gt^n_1 (t_n,H_1) \sqrt{h}, \\
	H_3 &= \bar{V}^n_n + \gt^n_0 (t_n,H_1) h - \gt^n_1 (t_n,H_1) \sqrt{h}, \\
	H_4 &= \bar{V}^n_n + \gt^n_0 (t_n,H_1) h + \gt^n_1 (t_n,H_1) \sqrt{h} + e^{-\Delta L_2^n}\gt^n_1 (t_n+h,e^{\Delta L_2^n}H_2) \sqrt{h}, \\
 	H_5 &= \bar{V}^n_n + \gt^n_0 (t_n,H_1) h + \gt^n_1 (t_n,H_1) \sqrt{h} - e^{-\Delta L_2^n}\gt^n_1 (t_n+h,e^{\Delta L_2^n}H_2) \sqrt{h},
\end{align*}
\begin{multline*}
 	\bar{V}^n_{n+1}= \bar{V}^n_n + \gt^n_1 (t_n,H_1) \Delta W^n\\ +
 						 \Big[e^{-\Delta L_2^n}\gt^n_0 (t_n+h,e^{\Delta L_2^n}H_2) - e^{-\Delta L_3^n}\gt^n_0 (t_n+h,e^{\Delta L_3^n}H_3)\Big]\frac{\Delta Z^n}{2\sqrt{h}} \\
 						+ \Big[e^{-\Delta L_2^n}\gt^n_0 (t_n+h,e^{\Delta L_2^n}H_2) + 2\gt^n_0 (t_n,H_1) + e^{-\Delta L_3^n}\gt^n_0 (t_n+h,e^{\Delta L_3^n}H_3)\Big]\frac{h}{4}  \\
 						+ \Big[e^{-\Delta L_2^n}\gt^n_1 (t_n+h,e^{\Delta L_2^n}H_2) - e^{-\Delta L_3^n}\gt^n_1 (t_n+h,e^{\Delta L_3^n}H_3)\Big]\frac{(\Delta W^n)^2 - h}{4\sqrt{h}} \\
 						 + \Big[e^{-\Delta L_2^n}\gt^n_1 (t_n+h,e^{\Delta L_2^n}H_2) - 2\gt^n_1 (t_n,H_1) + e^{-\Delta L_3^n}\gt^n_1 (t_n+h,e^{\Delta L_3^n}H_3)\Big]\frac{\Delta W^n h - \Delta Z^n}{2h}\\
 						+\Big[e^{-\Delta L_4^n}\gt^n_1 (t_n+h,e^{\Delta L_4^n}H_4) - e^{-\Delta L_5^n}\gt^n_1 (t_n+h,e^{\Delta L_5^n}H_5) - e^{-\Delta L_2^n}\gt^n_1 (t_n+h,e^{\Delta L_2^n}H_2) \\
 						+ e^{-\Delta L_3^n}\gt^n_1 (t_n+h,e^{\Delta L_3^n}H_3)\Big] \frac{1}{4h}\left[ \frac{1}{3}(\Delta W^n)^2 - h \right] \Delta W^n
\end{multline*}
and
\[
Y_{n+1} = e^{\Delta L^n} \bar{V}^n_{n+1}.
\]
\subsection{Platen weak order 2.0 SL scheme} \label{sec:Platen2}
Following the same steps as above, for the explicit order $2$ weak scheme by Platen for $M=1$ \cite[Eq. 15.1.1]{kloeden99nso} as underlying scheme we obtain
\begin{center}
\begin{tabular}{ll}
	$\Delta L_1^n = 0$  & $\Delta L_2^n = (A_0 - \frac{1}{2} A_1^2) h + A_1 \Delta W^n$ \\[1em]
	$\Delta L_3^n = (A_0 - \frac{1}{2} A_1^2) h + A_1 \sqrt{h}$ & $\Delta L_4^n = (A_0 - \frac{1}{2} A_1^2) h - A_1 \sqrt{h}$ \\[1em]
	$\Delta L^n = (A_0 - \frac{1}{2} A_1^2) h + A_1 \Delta W^n $ &
\end{tabular}
\end{center}

\noindent and thus the corresponding Platen weak order 2.0 SL scheme is given by
\begin{align*}
	H_1 &= \bar{V}^n_n, \\
	H_2 &= \bar{V}^n_n + \gt^n_0 (t_n,H_1) h + \gt^n_1 (t_n,H_1) \Delta W^n, \\
	H_3 &= \bar{V}^n_n + \gt^n_0 (t_n,H_1) h + \gt^n_1 (t_n,H_1) \sqrt{h}, \\
	H_4 &= \bar{V}^n_n + \gt^n_0 (t_n,H_1) h - \gt^n_1 (t_n,H_1) \sqrt{h}, \\
 	\\
 	\bar{V}^n_{n+1} &= \bar{V}^n_n + \Big[e^{-\Delta L_2^n}\gt^n_0 (t_n+h,e^{\Delta L_2^n}H_2) + \gt^n_0 (t_n,H_1)\Big] \frac{h}{2}\\
 						&+ \Big[e^{-\Delta L_3^n}\gt^n_1 (t_n+h,e^{\Delta L_3^n}H_3) + 2\gt^n_1 (t_n,H_1) + e^{-\Delta L_4^n}\gt^n_1 (t_n+h,e^{\Delta L_4^n}H_4)\Big] \frac{\Delta W^n}{4} \\
 						&+ \Big[e^{-\Delta L_3^n}\gt^n_1 (t_n+h,e^{\Delta L_3^n}H_3) - e^{-\Delta L_4^n}\gt^n_1 (t_n+h,e^{\Delta L_4^n}H_4)\Big]\frac{(\Delta W^n)^2 - h}{4\sqrt{h}} 	\\ \\
 	Y_{n+1} &= e^{\Delta L^n} \bar{V}^n_{n+1}.
\end{align*}
}
\end{document}